\documentclass{amsart}
\usepackage{amsmath}
\usepackage{amssymb}
\usepackage{tikz-cd}
\usepackage{mathrsfs}
\usepackage{bm}
\usepackage{hyperref}

\numberwithin{equation}{section}

\newtheorem{thm}[equation]{Theorem}
\newtheorem{lem}[equation]{Lemma}
\newtheorem{prop}[equation]{Proposition}
\newtheorem{cor}[equation]{Corollary}

\newtheorem{introthm}{Theorem}

\theoremstyle{definition}
\newtheorem{defn}[equation]{Definition}

\theoremstyle{remark}
\newtheorem{ex}[equation]{Example}
\newtheorem{nex}[equation]{Non-example}
\newtheorem{const}[equation]{Construction}
\newtheorem{rem}[equation]{Remark}
\newtheorem{warn}[equation]{Warning}

\DeclareMathOperator{\MSet}{ \boldsymbol{\mathcal{M}}\mathbf{-Set} }

\DeclareMathOperator{\EMSSet}{ \boldsymbol{\mathnormal{E}}\boldsymbol{\mathcal{M}}\mathbf{-SSet} }

\DeclareMathOperator{\EMSSetm}{ \boldsymbol{\mathnormal{E}}\boldsymbol{\mathcal{M}}\mathbf{-SSet}^{\mu} }

\DeclareMathOperator{\EMSSett}{ \boldsymbol{\mathnormal{E}}\boldsymbol{\mathcal{M}}\mathbf{-SSet}^{\tau} }

\DeclareMathOperator{\M}{ \mathcal{M} }
\DeclareMathOperator{\EM}{ \mathnormal{E}\mathcal{M} }
\DeclareMathOperator{\Inj}{ Inj }
\DeclareMathOperator{\maps}{ maps }
\DeclareMathOperator{\colim}{ colim }
\DeclareMathOperator{\CMon}{ CMon }
\DeclareMathOperator{\im}{ im }
\DeclareMathOperator{\incl}{ incl }

\DeclareMathOperator{\id}{ id }

\DeclareMathOperator{\op}{ op }
\DeclareMathOperator{\supp}{ supp }
\newcommand{\Ho}{\textup{Ho}}
\newcommand{\ppo}{\mathbin{\raise.25pt\hbox{\scalebox{.67}{$\square$}}}}
\newcommand{\cat}[1]{\textbf{\textup{#1}}}
\let\del=\partial
\let\phi=\varphi

\begin{document}
\title[Simplicial $*$-modules and mild actions]{Simplicial $\bm*$-modules and mild actions}
\author{Tobias Lenz}\address{\null\hskip-\parindent T.L.: Mathematisches Institut, Rheinische Friedrich-Wilhelms-Universit\"at Bonn, Endenicher Allee 60, 53115 Bonn, Germany\hskip0pt plus 5pt \&\hskip 0pt plus 5pt Mathematical Institute, University of Utrecht, Budapestlaan 6, 3584 CD Utrecht, The Netherlands (\textit{current address})}
\author{Anna Marie Schröter}
\address{\null\hskip-\parindent A.M.S.: Mathematisches Institut, Rheinische Friedrich-Wilhelms-Universit\"at Bonn, Endenicher Allee 60, 53115 Bonn, Germany}

	\begin{abstract}
		We develop an analogue of the theory of \emph{$*$-modules} in the world of simplicial sets, based on actions of a certain simplicial monoid $E\mathcal M$ originally appearing in the construction of global algebraic $K$-theory.

		As our main results, we show that strictly commutative monoids with respect to a certain box product on these simplicial $*$-modules yield models of equivariantly and globally coherently commutative monoids, and we give a characterization of simplicial $*$-modules in terms of a certain \emph{mildness} condition on the $E\mathcal M$-action, relaxing the notion of \emph{tameness} previously investigated by Sagave--Schwede and the first author.
	\end{abstract}

	\maketitle

	\section*{Introduction}
	Homotopy theoretic analogues of algebraic structures like groups and rings have been studied by algebraic topologists for more than fifty years. One of the earliest milestones was May's \emph{Recognition Theorem} \cite{may-gil} giving a `homotopy algebraic' description of $n$-fold loop spaces using the language of operads; more precisely, he introduced \emph{$E_n$-operads} and showed that $n$-fold loop spaces are equivalent to the corresponding \emph{(grouplike) $E_n$-algebras}. In the limit case $n=\infty$ this yields an equivalence between connective spectra and \emph{grouplike $E_\infty$-algebras}, which make precise the idea of a `group' whose multiplication is associative, commutative, and unital only up to a coherent system of homotopies. On top of its intrinsic interest, this result is central in that allows the construction of spectra and hence cohomology theories from data of a more algebraic, combinatorial, or categorical flavour.

	Accordingly, the case $n=\infty$ has received special attention, and various other models of such `coherently commutative groups' and more generally `coherently commutative monoids' have been established over the years. In particular, shortly after May's work, Segal introduced his \emph{special $\Gamma$-spaces} \cite{segal} providing an alternative approach to the subject. Furthermore, various `{ultra-commutative}' models have been studied in more recent years \cite{may-operadic, blum, sagave-schlichtkrull, schwede-book, schwede,lenzGglobal}: as their common trait, these model `coherently commutative monoids' as \emph{strictly} commutative monoids in a suitable symmetric monoidal base category; note that this base category is necessarily different from topological spaces with the cartesian product as it has been long known \cite{dold-thom} that the na\"ive notion of a \emph{topological abelian monoid} is too rigid to account for all connective spectra.

	One advantage of the ultra-commutative models over the other approaches is that they easily generalize to the world of \emph{$G$-equivariant homotopy theory} for finite groups $G$: namely, while the correct notion of `$G$-equivariantly coherently commutative monoids' (corresponding to \emph{genuine} equviariant cohomology theories, coming with a form of duality for all $G$-manifolds) is more subtle than an $E_\infty$-algebra with a $G$-action and required the introduction of \emph{genuine $G$-$E_\infty$-algebras} \cite{guillou-may}, and while similarly the $G$-equivariant generalization of Segal's theory relies on a non-obvious strengthening of the specialness condition \cite{shimakawa}, for many of the ultra-commutative models the most na\"ive approach actually turns out to work: $G$-objects in the underlying model represent all $G$-equivariantly coherently commutative monoids \cite{lenzGglobal}.

	In a similar vein, the ultra-commutative models often lend themselves to the study of \emph{global equivariant} phenomena. Namely, many interesting equivariant cohomology theories like $K$-theory or complex bordism turn out to exist in a compatible way for all finite groups $G$, leading to the notion of a \emph{global spectrum} \cite{schwede-book}. Various notions of ultra-commutative monoids then actually already form models of `globally coherently commutative monoids,' and in particular they account for all connective global spectra \cite{schwede-book,lenzGglobal}. Note that for $G\not=1$ not every $G$-equivariant spectrum is part of such a `globally compatible family,' and in particular global homotopy theory is not a generalization of $G$-equivariant homotopy theory. Instead, one can consider \emph{$G$-global homotopy theory} as the joint synthesis of the global and $G$-equivariant approaches, which is in particular a natural home of phenomena like Real $K$-Theory and Real bordism. Again there is a corresponding notion of `$G$-globally coherently commutative monoids' accounting for all connective $G$-global spectra, and these can in particular be simply modelled by $G$-objects in suitable categories of ultra-commutative monoids \cite{lenzGglobal}.

	\subsection*{Two `box products\rlap{\kern1pt'}}
	In the present paper we are concerned with two of the ultra-commutative approaches and their relationship to each other:

	One of the oldest ultra-commutative models of coherently commutative monoids is based on the category $\cat{$\bm{\mathcal L}$-Top}$ of topological spaces with an action of the topological monoid $\mathcal L$ of linear isometric self-embeddings of $\mathbb R^\infty$. This category comes with a \emph{box product} $\boxtimes_{\mathcal L}$ that is associative and symmetric, but in general not unital. Nevertheless, we still have a natural comparison map $X\boxtimes_{\mathcal L}*\to X$ for any $\mathcal L$-space $X$, and Blumberg, Cohen, and Schlichtkrull \cite{blum} defined a \emph{$*$-module} as an $\mathcal L$-space for which this map is an isomorphism. Restricted to the full subcategory of $*$-modules, the box product then indeed gives a symmetric monoidal structure, and the resulting category of commutative monoid objects provides the desired model. This approach can be seen as an unstable analogue of the theory of \emph{$S$-modules} and \emph{(commutative) $S$-algebras} developed in \cite{Elmendorf}. Moreover, it is closely related to May's original operadic models: the monoid $\mathcal L$ is actually given by the $1$-ary operations of a well-studied $E_\infty$-operad (the \emph{linear isometries operad}, denoted by the same symbol), and the box product is an instance of a so-called \emph{operadic product}, which in particular allows us to view the resulting commutative monoids (on the model level) as $\mathcal L$-algebras with extra properties \cite{may-operadic, blum}.

	More recently, a different approach based on the notion of \emph{parsummability} has been studied in \cite{schwede-k-theory,lenzGglobal}. Here the basic objects of study are simplicial sets with an action of a certain simplicial monoid $E\mathcal M$ built from self-injections of the countably infinite set $\{1,2,\dots\}$. These $E\mathcal M$-simplicial sets are required to satisfy an additional technical condition called \emph{tameness}, which roughly says that for any fixed simplex the action of an injection is determined by what it does on a large enough finite set. The tame $E\mathcal M$-simplicial sets then again come with a box product which is now defined as a certain subcomplex of the cartesian product. The resulting commutative monoid objects are called \emph{parsummable simplicial sets}, where the neologism `parsummable' is an abbreviation for `partially summable,' motivated by the fact that these are now in particular simplicial sets equipped with a \emph{strictly} associative, commutative, and unital, but only \emph{partially defined} addition operation. As shown in \cite{lenzGglobal}, these parsummable simplicial sets form a model of globally (and hence in particular also of ordinary) coherently commutative monoids, and more generally parsummable simplicial sets with $G$-action model $G$-globally and $G$-equivariantly coherently commutative monoids.

	\subsection*{New results}
	The monoid $\mathcal L$ is an intrinsically topological object, and in particular while one could build a simplicial monoid from it by just applying the standard singular complex functor, one probably shouldn't expect to obtain natural examples of simplicial sets with such an action via combinatorial or algebraic methods.

	Instead, the simplicial monoid $E\mathcal M$ is a natural combinatorial stand-in for $\mathcal L$, and it similarly extends to a well-known $E_\infty$-operad (in simplicial sets this time), the \emph{injections operad} $\mathcal I$. There is then again a suitable associative, commutative, but non-unital box product on $E\mathcal M$-simplicial sets, leading to the titular simplicial $*$-modules and their commutative monoid objects, which in analogy with \cite{Elmendorf} we call \emph{commutative $*$-algebras}. As our first main result we prove:

	\begin{introthm}[See Propositions~\ref{prop:*-mod-quillen} and~\ref{prop:*-mod-equiv}, Theorem~\ref{thm:*-alg-model}]\label{introthm:*-alg}
		For any finite group $G$, the categories $\cat{$\bm G$-$\bm*$-Mod}$ and $\cat{$\bm G$-$\bm *$-Alg}$ of $*$-modules and commutative $*$-algebras with $G$-action admit $G$-global model structures, modelling $G$-global spaces and $G$-globally coherently commutative monoids, respectively.
	\end{introthm}

	This in particular implies that commutative $G$-$*$-algebras also form a model of \emph{$G$-equivariantly} coherently commutative monoids, see Corollary~\ref{cor:g-*-equivariant}.

	As our second main result, we show how this approach can be reinterpreted in line with the parsummable philosophy. More precisely, we introduce a new \emph{mildness} condition for $E\mathcal M$-actions relaxing the notion of tameness, and we show that the category of mild $E\mathcal M$-simplicial sets admits a natural box product extending the box product of tame $E\mathcal M$-simplicial sets. With these we then have:

	\begin{introthm}[See Theorems~\ref{thm:starModEq} and~\ref{thm:boxprodOperadic}]\label{introthm:*-mod}
		An $E\mathcal M$-simplicial set is a $*$-module if and only if it is mild. Moreover, the box product of mild $E\mathcal M$-simplicial sets agrees with the operadic product associated to the injections operad $\mathcal I$.
	\end{introthm}

	Under these identifications, the equivalence from Theorem~\ref{introthm:*-alg} is then in fact simply given by the inclusion of tame into mild $E\mathcal M$-simplicial sets. Using the close connection between commutative $*$-algebras and $\mathcal I$-algebras, we moreover sketch (Remark~\ref{rk:operads-again}) how this gives a new and more direct proof of the equivalence between the operadic and parsummable approaches to $G$-globally and $G$-equivariantly coherently commutative monoids first established in \cite{lenzOperads}.

	\subsection*{Open questions}
	It is natural to ask how much of these results transfers to the world of $\mathcal L$-spaces. For example, it would be desirable to have a more intrinsic characterization of $*$-modules à la \cite{blum} in terms of an analogue of the mildness condition. While such a characterization indeed seems plausible at this point, a proof of this would require different techniques than the ones employed in the present paper and in particular would have to deal with pointset topological issues.

	Furthermore, one can ask whether also the $\mathcal L$-space approach allows to model globally and more generally $G$-globally coherently commutative monoids. For the trivial group, some results in this direction were obtained by Böhme \cite{boehme}, who in particular proved that $*$-modules are a model of global spaces. While he also provides a comparison result for non-commutative monoids, the (harder) commutative case as well as the situation for general $G$ are still open to our knowledge.

	\subsection*{Outline}
	In the main part of the paper, it will actually be more convenient to argue precisely the other way round than suggested above: namely, we will focus on the study of mild $E\mathcal M$-simplicial sets and their box product, and then use Theorem~\ref{introthm:*-mod} to recast these results in the language of $*$-modules and their operadic product. In more detail:

	In Section~\ref{sec:M-sets} we recall the discrete monoid $\mathcal M$ underlying $E\mathcal M$ and establish basic results about the combinatorics of $\mathcal M$-sets that we will need throughout. Section~\ref{sec:*-mod} is then concerned with the category of mild $E\mathcal M$-simplicial sets and its symmetric monoidal structure, culminating in the proof of Theorem~\ref{introthm:*-mod}. Finally in Section~\ref{sec:*-alg} we introduce commutative $*$-algebras and prove Theorem~\ref{introthm:*-alg}.

	\subsection*{Acknowledgements}
	The first two sections are based on the second author's master thesis written under the supervision of Stefan Schwede and the first author, which also contained the non-equivariant analogue of Theorem~\ref{introthm:*-alg}.

	The authors would like to thank Stefan Schwede for helpful discussions and feedback.

	\section{\texorpdfstring{${\M} $}{M}-sets}\label{sec:M-sets}
	Before we can come to the simplicial monoid $E\mathcal M$, we first have to understand its discrete analogue:

	\begin{defn}
		We define \emph{$ \M $} as the monoid of injective self-maps of the countably infinite set $ \omega = \{1,2,3,\dots\} $. An \emph{$ \M $-set} is a set with left $ \M $-action.
		We write $ \MSet $ for the category of $ \M $-sets with the $ \M $-equivariant maps.
	\end{defn}

	$\mathcal M$-actions (on abelian groups) were originally investigated in the context of the na\"ive homotopy groups of symmetric spectra \cite{shipley-thh, schwede-symmetric}. More recently, $\mathcal M$-actions on simplicial sets were studied non-equivariantly by Sagave and Schwede \cite{schwede}, who exhibited a close connection to \emph{$\mathcal I$-spaces} in the sense of \cite{sagave-schlichtkrull}, and from a global point of view by the first author \cite[Section~1.2]{lenzGglobal}.

	In general, monoid actions behave in many ways pathological and unfortunately $\mathcal M$-actions are no exception from this. However, many examples of interest satisfy an additional technical condition called \emph{tameness}, causing them to behave much more like objects with an action of a group. Below, we will recall the notion of tameness and introduce a relaxation of it, that we call \emph{mildness}. This new mildness condition will still be enough to recover many of the pleasant properties of tame $\M$-actions established in \cite{schwede}, but the subcategory of mild $\mathcal M$-objects will be better behaved than the subcategory of tame objects (cf.~Theorem~\ref{thm:mild-presentable}).

	For this we first introduce:

	\begin{defn}\label{def:supported-M}
		Let $ X $ be an $ \M $-set and $ A\subset \omega $. We write \emph{$ \M_A $} for the submonoid of $ \M $ given by the injections fixing $ A $ elementwise. We say that $ x\in X $ is \emph{supported on $ A $} if it is fixed by $\M_A$.
	\end{defn}

	\begin{lem}
		\label{lem:agreeSupp}
		Let $ X $ be an $ \M $-set, let $A\subset\omega$ be co-infinite (i.e.~$|\omega\setminus A|=\infty$), and let $x\in X$ be supported on $A$. Then the following hold:
		\begin{enumerate}
			\item If $ f,g\in \M $ agree on $ A $ then $ f.x = g.x $.\label{item:aS-through-A}
			\item For every $ f\in\M $ the element $ f.x $ is supported on $ f(A) $.\label{item:aS-action}
			\item If $ f.x $ is supported on $ f(A') $ for some $ A'\subset A $, then $ x $ is supported on $ A' $.\label{item:aS-preimage}
		\end{enumerate}
	\end{lem}
	\begin{proof}
		For finite $A$, the first two statements appear as \cite[Proposition 2.5]{schwede}; the proof in our situation is completely analogous, so we will only prove the first one. For this we observe that since $ A $ has infinite complement, we can find a \emph{bi}jection $ h \in \M $ such that $ h $ agrees with $ f $ and $ g $ on $ A $. Then $ h^{-1}f $ and $ h^{-1}g $ fix $ A $ elementwise. Hence $ f.x = h.h^{-1}f.x = h.x = h.h^{-1}g.x = g.x $.

		For the third statement we again take a bijection $ h\in \M $ such that $ h^{-1}f $ fixes $ A $ elementwise, so that $x=(h^{-1}f).x= h^{-1}.f.x$. Then the second statement implies that $x $ is supported on $ h^{-1}(f(A')) = h^{-1}f(A') = A' $ as claimed.
	\end{proof}

	The above notion also behaves well under finite intersections:

	\begin{prop}\label{prop:capSupp}
		Let $ X $ be an $ \M $-set, let $A,B\subset\omega$ co-infinite, and let $ x\in X $ be supported on both $ A $ and $ B $. Then $ x $ is supported on $ A\cap B $.
	\end{prop}

	Beware that the proposition is not true without the co-infiniteness assumption, see \cite[discussion after Definition~2.4]{schwede}.

	\begin{proof}
		Let $ f\in \M_{A\cap B} $. Then we have
		\begin{align*}
		(A^c\setminus f(A))\cup (A^c\setminus f(B))
		&= A^c\setminus(f(A)\cap f(B))\\
		&= A^c\setminus f(A\cap B)\\
		&= A^c\setminus (A\cap B)\\
		&= A^c
		\end{align*}
		since $ f $ is injective and fixes $ A\cap B $ elementwise. Thus, at least one of $ A^c\setminus f(A) $ and $ A^c\setminus f(B) $ is infinite because $ A^c $ is so. We will consider these two cases separately.

		\medskip
		\textit{Case 1.} If $ |A^c\setminus f(A)|=\infty $, we first construct a map $ f_1 \in \M $ with $ f_1 |_A = f|_A $ and $ {f_1}(A^c)\subset A^c $ as follows: we choose an injection $ \phi: A^c \to A^c\setminus f(A) $ (which is possible since the right hand side is assumed to be infinite), and we define $ f_1 (a) = f(a) $ for $ a\in A $ and $ f_1 (b) = \phi(b) $ for $ b \in A^c $; this is again injective as $ f|_A $ and $ \phi $ are injective with $ \im(\phi) \cap f(A) = \emptyset $.
		Moreover, $ f_1(A^c) \subset A^c $ by definition of $ \phi $.

		We then define $f_2\colon\omega\to\omega$ via
		\begin{equation*}
			f_2(a)=\begin{cases}
				a & \text{if }a\in A\\
				f_1(a) & \text{otherwise}.
			\end{cases}
		\end{equation*}
		Clearly, $ f_2|_A $ and $f_2|_{A^c}$ are injective, and $f_2(A)\cap f_2(A^c)= A\cap f_1(A^c)= \emptyset $; thus, $ f_2 $ is injective. If now $ b\in B\setminus A\subset A^c $ then $ f_2(b) = f_1(b) $, while for $ b\in A\cap B $ we have $ f_2(b) = b = f_1(b) $; thus, $ f_2 $ and $ f_1 $ agree on $ B $. By Lemma~\ref{lem:agreeSupp}-(\ref{item:aS-through-A}) we therefore get $f.x=f_1.x=f_2.x=x$ as claimed.

		\medskip
		\textit{Case 2.} If $ |A^c\setminus f(B)|=\infty $, we first construct a map $ g_1\in \M $ with $ g_1|_B = f|_B $ and $g_1(B^c)\subset A^c$ together with an injection $ \psi\colon B\setminus A \to \im(g_1)^c \cap A^c$ as follows: we choose an injection $\hat\psi:{B^c} \sqcup {(B\setminus A)} \to A^c\setminus f(B) $, and we define $ g_1(b) = f(b) $ for $ b\in B $ and $ g_1(a)=\hat\psi(a) $ for $ a \in B^c $. Then $ g_1 $ is injective by the same argument as above, agrees with $ f $ on $ B $ by construction, sends $B^c$ to $A^c$ because $\hat\psi$ does, and $\psi\mathrel{:=} \hat\psi|_{B\setminus A}$ has image disjoint from $\im(g_1)$ by injectivity of $\hat\psi$.

		Next, we construct $ g_2 \in \M $ with $ g_2|_A = g_1|_A $ and $ g_2^{-1} (A) \subset A $ as follows: we set $ g_2(b) = \psi(b) $ for $ b\in B\setminus A $ and $ g_2(c) = g_1(c) $ for $ c\notin B\setminus A $, which is clearly injective and agrees with $g_1$ on $A$. If now $ g_2(d) \in A $, then $d\notin B\setminus A$ as $g_2(B\setminus A)=\psi(B\setminus A)\subset A^c$, but also $d\notin B^c$ by definition of $g_1$, so $d\in (B\setminus A)^c\cap B=A\cap B\subset A$.

		Finally, we define $ g_3$ via
		\begin{equation*}
			g_3(a)=\begin{cases}
				a & \text{if }a\in A\\
				g_2(a) & \text{otherwise}.
			\end{cases}
		\end{equation*}
		This is injective as $ g_2 $ is injective and $ g_2^{-1}(A)\subset A$. If now $ a\in A\cap B $, then $ g_3(a) = a = g_2(a) $, while for $ b\in B\setminus A $ we have $ g_3(b) = g_2(b) $ by definition of $ g_3 $. Hence, $ g_3|_B=g_2|_B $ and altogether $f.x=g_1.x=g_2.x=g_3.x=x$ as claimed.
	\end{proof}

	\begin{rem}
		If  $ A $ and $ B $ are finite, the proposition can be proven by factoring $ f $ as a composition of maps in $ \M $ each of which fixes $ A $ or $ B $ elementwise, see \cite[Proposition 2.3]{schwede}. In our case such a decomposition of $ f $ is not possible anymore, which is the reason the above proof is considerably more involved.

		As a concrete counterexample, let $ A\subset \omega $ be both infinite and co-infinite. Then any map fixes $\emptyset=A\cap A^c$ pointwise; we claim however that not every map factors as a composite of maps fixing $A$ or $A^c$. For this, let $f\in\M$ with $f(A)\not\subset A$, e.g.~$f(x)=x+1$. If now $ \sigma \in \M_A,\tau\in\M_{A^c}$, then in particular, $ \sigma(A) = A $ and $ \tau(A^c) = A^c $, whence also $ \tau(A) \subset A $ by injectivity. Thus, any composition of maps in $ \M_A \cup \M_{A^c} $ maps $ A $ into $ A $, and $f$ cannot be factored accordingly.
	\end{rem}

	\begin{warn}
		If $X$ is an $\M$-set and $x\in X$ is supported on some \emph{finite} set, then the proposition implies that $x$ is supported on a unique minimal finite (or equivalently co-infinite) set $\supp(x)$, called the \emph{support of $x$} \cite[Definition~2.4]{schwede}. In general however, this is no longer true, i.e.~the notion of support is not compatible with \emph{infinite} intersections as the following example shows:

		Write $ A_n\mathrel{:=}\{k\in \omega : k \text{ even and }k\geq 2n \} $, so that each $A_n$ is co-infinite and $\bigcap_{n\ge0}A_n=\emptyset$. We then define $X\mathrel{:=}{\M} / {\sim} $ where $\sim$ is the equivalence relation \[f \sim g \mathrel{:\Leftrightarrow} \text{$f(2x)=g(2x)$ for almost all $x\in\omega$} ;\] clearly the usual $\mathcal M$-action on itself descends to $X$.

		If now $f\in\mathcal M_{A_n}$ for some $n$, then $f.[\id]=[f]$ agrees with $\id$ by definition of the equivalence relation, i.e.~$[\id]\in X$ is supported on all $A_n$. However, it is not supported on the empty set (i.e.~$\mathcal M$-fixed): $f(x)=x+1$ does not fix any even numbers, and in particular $f.[\id]=[f]\not=[\id]$.
	\end{warn}

	\begin{defn}
		Let $ X $ be an $ \M $-set.
		\begin{enumerate}
			\item $X$ is called \emph{tame} if for every $x\in X$ there exists a \emph{finite} set $A\subset\omega$ such that $x$ is supported on $A$.
			\item $X$ is called \emph{mild} if for every $x\in X$ there exists a \emph{co-infinite} set $A\subset\omega$ such that $x$ is supported on $A$.
		\end{enumerate}
	\end{defn}

	Clearly, each tame $ \M $-set is also mild.

	\begin{ex}\label{ex:injA}
		For any finite set $A\subset\omega$, the set $\Inj(A,\omega)$ of injections $A\to\omega$ with $\M$-action via postcomposition is tame, hence in particular mild. More precisely, any injection $i\colon A\to\omega$ is supported on its image $i(A)$.
	\end{ex}

	\begin{nex}
		The $\M$-set $\M$ itself is not mild: surjections are only fixed by the identity, hence in particular not supported on any co-infinite set.
	\end{nex}

	\begin{defn}
		Let $ X $ be an $ \M $-set. We define
		\begin{align*}
			X^\tau&=\{x\in X : \text{$x$ is supported on a finite set $A\subset\omega$}\}\\
			X^\mu&=\{x\in X : \text{$x$ is supported on a co-infinite set $A\subset\omega$}\}.
		\end{align*}
	\end{defn}

	By Lemma \ref{lem:agreeSupp}-(\ref{item:aS-action}) these are $\mathcal M$-subsets of $X$; clearly, $ X^{\mu} $ is then the maximal mild $ \M $-subset of $ X $ and $ X^{\tau} $ is its maximal tame $ \M $-subset.

	\begin{ex}\label{ex:mildM}
		We claim that the $ \M $-set $ \M^{\mu} $ is the set of all injections $ u\in\M $ whose image has infinite complement.
		Indeed, any $ u\in\M $ is clearly supported on its image; conversely, if $ A\subset\omega $ with infinite complement such that $ \im(u)\not\subset A $, then we can pick $b\in\im(u)\setminus A$ and $f\in\M_A$ missing $b$, in which case $fu\not=u$.

		Note that the same argument also shows that $\M^\tau=\emptyset$; in particular $\M^\mu$ is an example of a mild $\M$-set that is not tame.
	\end{ex}

	We record two further properties of mild $\mathcal M$-sets for later use. Firstly, we have the following generalization of \cite[Proposition~2.7]{schwede}:

	\begin{lem}
		\label{lem:injAct}
		Let $ X $ be a mild $ \M $-set and $ f\in\M $. Then the action of $ f $ is an injective map $ X\to X $.
	\end{lem}
	\begin{proof}
		Let $ x,y\in X $ with $ f.x = f.y $, and pick co-infinite sets $A_x, A_y$ with $x$ supported on $A_x$ and $y$ supported on $A_y$. Then $ f.x = f.y $ is supported on $ f(A_x)\cap f(A_y) = f(A_x\cap A_y) $ by Lemma \ref{lem:agreeSupp}-(\ref{item:aS-action}) together with Proposition \ref{prop:capSupp}. As an upshot, both $x$ and $y$ are supported on the \emph{same} co-infinite set $A\mathrel{:=} A_x\cap A_y$ by Lemma~\ref{lem:agreeSupp}-(\ref{item:aS-preimage}). We now simply pick a bijective $h\in\mathcal M$ agreeing with $f$ on $A$. Then $x=h^{-1}f.x=h^{-1}f.y=y$ as claimed.
	\end{proof}

	Finally, the same argument as in the tame case \cite[Lemma~1.3.13]{lenzGglobal} shows:

	\begin{lem}\label{lemma:complement}
		Let $X$ be a mild $\mathcal M$-set and let $Y\subset X$ be an $\mathcal M$-subset. Then its complement $X\setminus Y$ is again an $\mathcal M$-subset.\qed
	\end{lem}

	Beware that this is not true for general $\mathcal M$-sets, see the discussion after \emph{loc.~cit.}

	\section{Simplicial \texorpdfstring{$*$}{*}-modules}\label{sec:*-mod}
	\subsection{\texorpdfstring{$\bm{E\mathcal M}$}{EM}-simplicial sets}
	In this section, we will discuss $ \EM $-simplicial sets and the analogues of notions of tameness and mildness for them. Moreover, we will introduce a box product for mild $E\mathcal M$-simplicial sets generalizing \cite[Definition~2.1.8]{lenzGglobal} and reexpress it as a suitable operadic product.

	\begin{defn}
		We write \emph{$ \EM $} for the simplicial monoid given in degree $ n $ by $ (\EM)_n = \M^{1+n} \cong \maps(\{0,\dots,n\},\M) $, with pointwise multiplication and structure maps via precomposition. An \emph{$ \EM $-simplicial set} is a simplicial set with left $ \EM $-action.
		By $ \EMSSet $ we denote the category of $ \EM $-simplicial sets together with $ \EM $-equivariant simplicial maps.
	\end{defn}

	\begin{ex}\label{ex:EX}
		Let $ X $ be any $ \M $-set. Then we get an $ \EM $-simplicial set $ EX $ given in degree $ n $ by $ EX_n = X^{1+n} \cong \maps(\{0,\dots,n\},X) $, with structure maps again given by precomposition and with pointwise $ \EM $-action. In particular, we have an $E\mathcal M$-simplicial set $E\Inj(A,\omega)$ for any set $A$, with $E\mathcal M$-action via postcomposition.
	\end{ex}

	\begin{defn}
		Let $ X $ be an $ \EM $-simplicial set, $ x\in X_n $, $ A \subset \omega $, and $ 0\leq k\leq n $. Then we say that $ x $ is \emph{$ k $-supported on $ A $} if $ i_k(u).x = x $ for all $ u\in \M_A $ where $ i_k\colon\mathcal M\to\mathcal M^{1+n} =(E\mathcal M)_n $ is the inclusion of the $ (1+k) $-th factor.
	\end{defn}

	\begin{ex}\label{ex:EX-supp}
		Let $X$ be an $\mathcal M$-set. Then an $n$-simplex $(x_0,\dots,x_n)$ of $EX$ is $k$-supported on $A\subset\omega$ if and only if $x_k$ is supported on $A$ in the sense of Definition~\ref{def:supported-M}.
	\end{ex}

	In order for the above notions to be useful, they should of course interact reasonably with the $E\mathcal M$-action and the simplicial structure maps. We begin with the following analogue of Lemma~\ref{lem:agreeSupp}-(\ref{item:aS-action}):

	\begin{lem} \label{lem:ksupp}
		Let $ X $ be an $ \EM $-simplicial set, $ x\in X_n $, and $ u_0, \dots ,u_n \in \M $. Then $ (u_0, \dots ,u_n).x $ is $ k $-supported on $ \im(u_k) $. Moreover, if $ x $ is $ k $-supported on some co-infinite set $ A $, then $ (u_0, \dots ,u_n).x $ is $ k $-supported on $ u_k (A) $.
	\end{lem}
	\begin{proof}
		For any $g\in\mathcal M$ we have
		\begin{equation}\label{eq:ksupp-rewrite}
			i_k(g).(u_0,\dots,u_n).x=(u_0,\dots,u_{k-1},gu_k,u_{k+1},\dots,u_n).x.
		\end{equation}
		If now $g\in\mathcal M_{\im(u_k)}$, then $gu_k=u_k$, proving the first statement.

		For the second statement we further rewrite the right hand side of $(\ref{eq:ksupp-rewrite})$ as $$(u_0,\dots,u_{k-1},1,u_{k+1},\dots,u_n).i_k(gu_k).x.$$ If now $g\in\mathcal M_{u_k(A)}$, then $gu_k|_A=u_k|_A$, so Lemma~\ref{lem:agreeSupp}-(\ref{item:aS-action}) shows that $i_k(gu_k).x=i_k(u_k).x$, and the claim follows.
	\end{proof}

	\begin{prop} \label{prop:f(k)supp}
		Let $ X $ be an $ \EM $-simplicial set, let $ f\colon[m]\to [n] $ in $ \Delta $, and let $ 0\leq k \leq m $. Assume $ x\in X_n $ is $ f(k) $-supported on some co-infinite $ A\subset \omega $. Then $ f^* x $ is $ k $-supported on $ A $.
	\end{prop}

	The proof of this is considerably harder. We begin with two preparatory results:

	\begin{prop} \label{prop:infiniteCompl}
		Let $ {A\subset \omega} $ co-infinite and $ {u_0,\dots, u_n \in \M} $ with $\bigcup_{k=0}^n u_k(A)$ co-infinite. Then there exists a $ \chi \in \M_A $ such that also $ \bigcup_{k=0}^n\im(u_k \chi) $ is co-infinite.
	\end{prop}
	\begin{proof}
		This is very similar to the case for finite $A$ proven in \cite[Proposition 1.3.18]{lenzGglobal}, so we will only indicate where the argument given there has to be adapted.

		Write $A'\mathrel{:=}\bigcup_{k=0}^nu_k(A)$. Arguing as in \emph{loc.~cit.} we can find strictly increasing chains $ B_0 \subsetneq B_1 \subsetneq \dots $ and $ C_0 \subsetneq C_1 \subsetneq \dots $ of subsets of the infinite sets $ \omega\setminus A $ and $ \omega \setminus A' $, respectively, such that for all $ j\geq0 $
		\begin{equation}\label{eq:C_jnB_j}
		C_j \cap \bigcup_{k=0}^n u_k (B_j) = \emptyset.
		\end{equation}

		We will now explain how this proves the proposition. We set $ B_\infty := \bigcup_{j=0}^\infty B_j $ and $ C_\infty := \bigcup_{j=0}^\infty C_j $. These are clearly infinite sets; moreover, $ u_k(B_\infty) \cap C_\infty=\emptyset $ for all $k$ by property \eqref{eq:C_jnB_j}, while $B_\infty\cap A=\emptyset$ by construction.

		We can then find an injection $\chi'\colon\omega\setminus A\to B_\infty$ as the right hand side is infinite, which we can then further extend via the identity of $A$ to $\chi\colon\omega\to A\cup B_\infty$; note that $\chi$ is again injective as $A\cap B_\infty=\emptyset$. We now simply decompose \[\bigcup_{k=0}^n \im(u_k \chi) = \underbrace{\bigcup_{k=0}^n u_k (A)}_{=A'} \cup \underbrace{\bigcup_{k=0}^n u_k (B_\infty)}_{=:B'}\] and note that both $A'$ and $B'$ are disjoint from the infinite set $C_\infty$ by construction and the above observation, respectively, so that the left hand side is co-infinite as desired.
	\end{proof}

	\begin{prop} \label{prop:equalInM/M_A}
		Let $ A\subset \omega $ co-infinite, and let $ (u_0,\dots,u_n)$, $ (v_0,\dots,v_n)\in \M^{1+n} $ such that $ u_k|_A = v_k|_A $ for $ k=0,\dots,n $ and $ \bigcup_{k=0}^n u_k (A) $ is again co-infinite. Then $ [u_0,\dots,u_n] = [v_0,\dots,v_n] $ in $ \M^{1+n}/\M_A $.
	\end{prop}
	\begin{proof}
		One argues precisely as in the case of finite $A$ {\cite[Proposition~1.3.19]{lenzGglobal}}, appealing to the previous proposition instead of \cite[Proposition~1.3.18]{lenzGglobal}.
	\end{proof}

	\begin{proof}[Proof of Proposition~\ref{prop:f(k)supp}]
		We write $ f^{-1}(f(k)) \mathrel{=:} \{k_0,k_1\dots,k_r\} $ with pairwise distinct $k_i$ such that $k_0=k$, and we let $ i\colon \M^{1+r} \to \M^{1+m}$ denote the homomorphism $i(u_0,\dots,u_r) = i_{k_0}(u_0) \dotsm i_{k_r}(u_r) $.
		We now consider the map of sets \[\alpha \colon \M^{1+r} \to X_m, \ (u_0,\dots,u_r)\mapsto i(u_0,\dots, u_r).f^* x.\]
		Then for all $ v\in \M_A $ we find
		\begin{align*}
		\alpha(u_0 v,\dots, u_r v) & = i(u_0 v,\dots,u_r v).f^*x\\
		& = i(u_0,\dots,u_r).i(v,\dots,v).f^*x\\
		& = i(u_0,\dots,u_r).f^*(i_{f(k)}(v).x)\\
		& = i(u_0,\dots,u_r).f^* x = \alpha(u_0,\dots ,u_r),
		\end{align*}
		so $ \alpha $ factors through $\tilde\alpha\colon \M^{1+r} / \M_A\to X_m$. For any $u\in\mathcal M_A$ the previous proposition therefore shows that $i_{k}(u).f^*x = \tilde\alpha[u,1,\dots,1]=\tilde\alpha[1,\dots,1]=f^*x$.
	\end{proof}

	\subsection{Tameness, mildness, and the box product}
	We now come to the analogues of the notions of tameness and mildness for $ \EM $-simplicial sets:

	\begin{defn}
		Let $ X $ be an $ \EM $-simplicial set, $n\ge 0$, and let $x\in X_n$. We call $x$ \emph{finitely supported} if for every $0\le k\le n$ there exists a finite $A_k\subset\omega$ such that $x$ is $k$-supported on $A$. Similarly, we will say that $x$ is \emph{co-infinitely supported} if for every $0\le k\le n$ there exists a co-infinite $A_k\subset\omega$ on which $x$ is $k$-supported.

		We call $ X $ \emph{tame} if all its simplices are finitely supported, and \emph{mild} if all its simplices are co-infinitely supported. We write $\cat{$\bm{E\mathcal M}$-SSet}^\tau$ and $\cat{$\bm{E\mathcal M}$-SSet}^\mu$ for the full subcategories spanned by the tame and mild $E\mathcal M$-simplicial sets, respectively.
	\end{defn}

	\begin{warn}
		Let $X$ be an $E\mathcal M$-simplicial set and let $x\in X_n$. If $A\subset\omega$ is a finite set, then one can show \cite[Lemma~2.1.7]{lenzGglobal} that $x$ is $k$-supported on $A$ for all $0\le k\le n$ if and only if it is supported on $A$ with respect to the \emph{diagonal} $\mathcal M$-action on $X_n$; in particular, the notion of tameness only depends on the underlying $\M$-simplicial set of $X$.

		This is \emph{not} true in the co-infinite setting; in particular, $E(\mathcal M^\mu)$ is a mild $E\mathcal M$-simplicial set, but the set $E(\mathcal M^\mu)_1$ of edges is not mild as an $\mathcal M$-set: for example, if $f(x)=2x, g(x)=2x+1$, then $(f,g)$ is not supported on any co-infinite $A$. However, Proposition~\ref{prop:equalInM/M_A} at least shows that conversely $X$ is mild in the above sense if each $X_n$ is mild as an $\mathcal M$-set.
	\end{warn}

	\begin{warn}
		Similarly to the previous counterexample, the cartesian product of two mild $E\mathcal M$-simplicial sets need not be mild anymore. While the category $\cat{$\bm{E\mathcal M}$-SSet}^\mu$ has all small limits (see Lemma~\ref{lemma:tau-mu} or Theorem~\ref{thm:mild-presentable}) and hence in particular binary products, we will always use `$\times$' to denote the usual cartesian product (i.e.~the categorical product in $\cat{$\bm{E\mathcal M}$-SSet}$) below.
	\end{warn}

	\begin{const}
		Let $X$ be an $E\mathcal M$-simplicial set. For any $n\ge 0$ we write
		\begin{align*}
			(X^\tau)_n &= \{\text{$x\in X_n$ finitely supported}\}\\
			(X^\mu)_n &= \{\text{$x\in X_n$ co-infinitely supported}\}.
		\end{align*}
	\end{const}

	\begin{lem}\label{lemma:tau-mu}
		\begin{enumerate}
			\item $X^\tau$ is a tame $E\mathcal M$-simplicial subset of $X$ and this defines a functor $ (-)^{\tau}\colon \EMSSet \to \EMSSett $ right adjoint to the inclusion. In particular, the full subcategory $ \EMSSett\subset\EMSSet $ is closed under all colimits.\label{item:tm.-tameColim}
			\item $X^\mu$ is a mild $E\mathcal M$-simplicial subset of $X$ and this construction defines a functor $(-)^{\mu}\colon\EMSSet\to\EMSSetm$ right adjoint to the inclusion. In particular, the full subcategory $\EMSSetm\subset\EMSSet$ is closed under all colimits.\label{item:tm-mildFunctor}
		\end{enumerate}
	\end{lem}
	\begin{proof}
		The first part is \cite[Corollary 1.3.23]{lenzGglobal} and its proof. For the second part, we note that $X^\mu$ is a simplicial subset by Proposition~\ref{prop:f(k)supp} and moreover closed under the $E\mathcal M$-action by Lemma~\ref{lem:ksupp}.

		If now $f\colon X\to Y$ is $E\mathcal M$-equivariant, and $x\in X_n$ is $k$-supported on some co-infinite $A\subset\omega$, then clearly also $f(x)$ is $k$-supported on $A$. It follows that $f\colon X\to Y$ restricts to $f^\mu\colon X^\mu\to Y^\mu$ and that for mild $X$ any such $f$ factors (necessarily uniquely) through $Y^\mu$, proving it is the desired right adjoint. As fully faithful left adjoints create colimits, it then follows formally that $\EMSSetm$ is closed under all colimits.
	\end{proof}

	\begin{ex}\label{ex:mildFunctor}
		Let $X$ be an $\mathcal M$-set. Then Example~\ref{ex:EX-supp} shows that $ (EX)^{\mu} = E{(X^{\mu})} $.
	\end{ex}

	We now come to the key feature of tame and mild $E\mathcal M$-simplicial sets as opposed to general $E\mathcal M$-simplicial sets: they come with an interesting symmetric monoidal structure given by the \emph{box product}.

	\begin{defn}\label{def:BoxProdT}
		Let $ X, Y $ be mild $ \EM $-simplicial sets. We define their \emph{box product} $ X \boxtimes Y $ in degree $n$ as the set of all simplices $ (x,y) \in X_n\times Y_n $ such that for every $0\le k\le n$ we can find $A_k,B_k\subset\omega$ with $A\cap B=\emptyset$ and $A\cup B$ co-infinite such that $x$ is $k$-supported on $A_k$ and $y$ is $k$-supported on $B_k$.
	\end{defn}

	\begin{ex}\label{ex:boxtimes-Inj}
		Let $A,B$ be (finitely or infinitely) countable sets. Combining Examples~\ref{ex:injA} and~\ref{ex:EX-supp}, an $n$-simplex $(u_0,\dots,u_n;v_0,\dots,v_n)$ of $E\Inj(A,\omega)^\mu\times E\Inj(B,\omega)^\mu$ belongs to the box product if and only if $\im(u_k)\cap\im(v_k)=\emptyset$ and $\im(u_k)\cup\im(v_k)$ is co-infinite for every $k=0,\dots,n$. Thus, the map $E\Inj(A\amalg B,\omega)^\mu\to E\Inj(A,\omega)^\mu\times E\Inj(B,\omega)^\mu$ restricting to the coproduct summands induces an isomorphism $E\Inj(A\amalg B,\omega)^\mu\cong E\Inj(A,\omega)^\mu\boxtimes E\Inj(B,\omega)^\mu$.
	\end{ex}

	If $X$ and $Y$ are tame, we recall that $x$ and $y$ are $k$-supported on unique minimal finite sets $\supp_k(x)$ and $\supp_k(y)$; the above condition then simplifies to demanding that $\supp_k(x)\cap\supp_k(y)=\emptyset$ for all $0\le k\le n$. In this form, the box product of tame $E\mathcal M$-simplicial sets first appeared as \cite[Construction~2.16]{lenzGHT}, also see~\cite[Definition~2.12]{schwede} for a similar construction in $\mathcal M$-(simplicial) sets. By \cite[Proposition~2.17]{lenzGHT}, the box product of tame $E\mathcal M$-simplicial sets defines a subfunctor of the cartesian product, giving rise to a simplicial symmetric monoidal structure on $\EMSSett$. We will now generalize this to the present situation:

	\begin{prop}\label{prop:boxprod}
		Let $X,Y$ be mild $\EM$-simplicial sets. Then $X\boxtimes Y\subset X\times Y$ is a mild $E\mathcal M$-simplicial subset. Moreover, this construction gives rise to a simplicial subfunctor $\EMSSetm\times\EMSSetm\to\EMSSetm$ of the cartesian product $\EMSSetm\times\EMSSetm\to\EMSSet$, and this subfunctor preserves simplicial tensors in each variable.
	\end{prop}
	\begin{proof}
		Proposition~\ref{prop:f(k)supp} shows that $ X\boxtimes Y $ is simplicial and Lemma \ref{lem:ksupp} implies that it is closed under the diagonal $E\mathcal M$-action on the cartesian product $X\times Y$. Furthermore, $ X\boxtimes Y $ is mild as $ (x,y) \in (X\boxtimes Y)_n $ is $ k $-supported on the co-infinite set $ A_k\cup B_k $ for $ A_k $ and $ B_k $ as in the definition.

		It is then clear that the box product defines a simplicial subfunctor of the cartesian product. Moreover, one immediately checks from the definitions that for any simplicial set $ K $ and mild $ \EM $-simplicial sets $ X,Y $ the associativity isomorphism $ K\times (X\times Y) \to (K\times X)\times Y $ restricts to $ K\times (X\boxtimes Y) \to (K\times X) \boxtimes Y $, and similarly for the second variable, i.e.~the box product preserves tensors in each variable.
	\end{proof}

	\begin{prop}\label{prop:symmon}
		The unitality, associativity, and symmetry isomorphisms of the cartesian product on $ \EMSSet $ restrict to corresponding isomorphisms for $ \boxtimes $ on $ \EMSSetm $. This makes $ \EMSSetm $ into a simplicial symmetric monoidal category with tensor product $ \boxtimes $ and unit the terminal $ \EM $-simplicial set.
	\end{prop}
	\begin{proof}
		We show that the associativity isomorphism restricts to an isomorphism, the arguments for the unitality and symmetry isomorphisms being similar but easier.

		We take $ ((x,y),z) \in (X\boxtimes Y)\boxtimes Z $ and show $ (x,(y,z))\in X\boxtimes (Y\boxtimes Z) $. By the definition of the box product we get sets $ A_k,B_k,C_k,D_k\subset \omega $ such that $ x $ is $ k $-supported on $ A_k $, $ y $ is $ k $-supported on $ B_k $, $ A_k \cap B_k = \emptyset $, $ |(A_k\cup B_k)^c| = \infty $, $ z $ is $ k $-supported on $ C_k $, $ (x,y) $ is $ k $-supported on $ D_k $, $ C_k \cap D_k = \emptyset $ and $ |(C_k\cup D_k)^c| = \infty $. Note that while we know that $(x,y)$ is also $k$-supported on $A_k\cup B_k$, the set $D_k$ provided by the definition of the box product is a priori unrelated to $A_k$ and $B_k$, and in particular $C_k$ need not be disjoint from $A_k$ and $B_k$; we fix this as follows:

		As $(x,y)$ is $k$-supported on $D_k$, so are $x$ and $y$; thus, Proposition \ref{prop:capSupp} shows that $ x $ is $ k $-supported on $A_k'\mathrel{:=} A_k\cap D_k $ and $ y $ is $ k $-supported on $B_k'\mathrel{:=} B_k\cap D_k $; in particular, $(x,y)$ is $k$-supported on $D_k'\mathrel{:=}A_k'\cup B_k'\subset D_k$. Replacing $A_k$ by $A_k'$, $B_k$ by $B_k'$, and $D_k$ by $D_k'$  we may then indeed assume without loss of generality that $ A_k \cup B_k = D_k $.

		But then $ B_k\cap C_k \subset D_k\cap C_k = \emptyset $ and $ |(B_k\cup C_k)^c|\geq |(D_k\cup C_k)^c| = \infty $, so $ (y,z) \in Y\boxtimes Z $ and $ (y,z) $ is $ k $-supported on $ B_k\cup C_k $. Finally we observe that $ A_k\cap (B_k\cup C_k) = \emptyset $ and $ |(A_k\cup B_k \cup C_k)^c|=|(D_k\cup C_k)^c| = \infty $ hence $ (x,(y,z))\in X\boxtimes (Y\boxtimes Z) $ as claimed. Symmetrically, one shows that $(x,(y,z))\in X\boxtimes(Y\boxtimes Z)$ if $((x,y),z)\in (X\boxtimes Y)\boxtimes Z$.

		The coherence relations for the resulting isomorphisms then directly follow from the relations for the cartesian symmetric monoidal structure on $\EMSSet$.
	\end{proof}

	\begin{rem}\label{rem:unbiasedBox}
		For mild $ \EM $-simplicial sets $ X_1,\dots,X_m $ the argument above gives by induction that for each bracketing of $ X_1\boxtimes\dots\boxtimes X_m $ the image of the natural embedding into $X_1\times\cdots\times X_m$ is given in degree $ n $ precisely by the elements $ (x_1,\dots, x_m) $ such that for each $ 0\leq k\leq n $ there exist pairwise disjoint sets $ A_k^{(i)}\subset \omega $ with $ \bigcup_{i=1}^m A_k^{(i)} $ co-infinite such that $x_i$ is $k$-supported on $A_i^{(k)}$. We write $X_1\boxtimes\cdots\boxtimes X_m$ (without any bracketing) for this subcomplex, and refer to it as the \emph{unbiased ($m$-fold) box product}.
	\end{rem}

	\begin{rem}
		By definition, the inclusion $\cat{$\bm{E\mathcal M}$-SSet}^\tau\hookrightarrow\cat{$\bm{E\mathcal M}$-SSet}^\mu$ is strict symmetric monoidal, and it follows formally that the right adjoint $(-)^\tau$ acquires a lax symmetric monoidal structure. Unravelling the definitions, this structure is given by the unique map $*\to *^\tau$ and the maps $(X^\tau)\boxtimes (Y^\tau)\to (X\boxtimes Y)^\tau$ induced by the inclusions; in particular, it is actually a \emph{strong} symmetric monoidal structure.
	\end{rem}

	\subsection{An operadic description of the box product} Next, we will recall the \emph{operadic product} \cite[Construction~2.1.18]{lenzGglobal} on all $E\mathcal M$-simplicial sets:

	\begin{const} \label{const:operadBox}
		For $ X_1, \dots ,X_n \in \EMSSet $ we define
		\begin{equation}\label{eq:operadBox}
		E{\Inj}(n\times \omega,\omega)\times_{\EM^n} X_1\times\dots\times X_n
		\end{equation}
		as the quotient of $ E{\Inj}(n\times \omega,\omega)\times X_1\times\dots\times X_n $ by the equivalence relation generated on $m$-simplices by
		\begin{align*}
		&(f_0,\dots,f_m;u_1.x_1,\dots,u_n.x_n)\\
		&\quad\sim(f_0\circ(u_1^{(0)}\amalg\cdots\amalg u_n^{(0)}),\dots,f_m\circ(u_1^{(m)}\amalg\cdots\amalg u_n^{(m)}); x_1,\dots,x_n)
		\end{align*}
		for all $u_i=(u_i^{(0)},\dots,u_i^{(m)})\in (E\mathcal M)_m$; put differently, this is the coequalizer of the maps
		\begin{equation*}
		E\Inj(n\times\omega,\omega)\times E\mathcal M^n\times X_1\times\cdots\times X_n\rightrightarrows E\Inj(n\times\omega,\omega)\times X_1\times\cdots X_n
		\end{equation*}
		induced by the left $E\mathcal M$-actions on the $X_i$'s and by the right action on $E\Inj(n\times\omega,\omega)$ via precomposition, respectively.

		The $E\mathcal M$-action on $E\Inj(n\times\omega,\omega)$ via postcomposition then gives an $E\mathcal M$-action on $(\ref{eq:operadBox})$, and this then becomes a functor $\EMSSet^{\times n}\to\EMSSet$ in the obvious way. This functor moreover comes with a natural map
		\begin{equation}\label{eq:operadNat}
		\Phi\colon E{\Inj}(n\times \omega,\omega)\times_{\EM^n} X_1\times\dots\times X_n \to X_1\times\dots\times X_n
		\end{equation}
		which is on $ m $-simplices given by \[ \Phi[f_0,\dots,f_m;x_1,\dots,x_n]= ((f_0\iota_1,\dots,f_m\iota_1).x_1,\dots,(f_0\iota_n,\dots,f_m\iota_n).x_n) \]
		where $ \iota_j(t) =(j,t) $ for $ j=1,\dots,n $.
	\end{const}

	\begin{rem}\label{rk:boxtimes-classical}
		The simplicial monoid $E\mathcal M$ can be identified with the $1$-ary operations of a certain simplicial operad $\mathcal I$, the \emph{injections operad}; here the $n$-ary operations $\mathcal I(n)$ are given by $E\Inj(n\times\omega,\omega)$, and the operadic structure maps are given by `composition and juxtaposition.' There is a general construction producing bifunctors from operads \cite{may-operadic}, and specializing this to the operad $\mathcal I$ yields the case $n=2$ of the above construction.

		On the other hand, in the topological world we can apply this general construction to the \emph{linear isometries operad} $\mathcal L$ with $\mathcal L(n)$ the space of linear isometric embeddings $\mathbb R[n\times\omega]\to\mathbb R[\omega]$ and similarly defined structure maps, yielding an operadic product on the category $\textbf{$\bm{\mathcal L}$-Top}$ of topological spaces with continuous action by the topological monoid $\mathcal L\mathrel{:=}\mathcal L(1)$. This operadic product and various variants of it have been studied extensively in the literature \cite{kriz-may, Elmendorf, blum, lind, boehme}, and it is also typically referred to as \emph{box product} and denoted by $\boxtimes$ or $\boxtimes_{\mathcal L}$. One can show that this box product is suitably associative and commutative, but it is \emph{not} unital in the $1$-categorical sense. Instead, there is a natural comparison map $X\boxtimes_{\mathcal L} *\to X$ for any $\mathcal L$-space $X$, and Blumberg, Cohen, and Schlichtkrull \cite[Definition~4.9]{blum} defined a \emph{$*$-module} as an $\mathcal L$-space for which this map is an isomorphism; already before that, a similar concept in the stable world had been studied under the name \emph{$S$-module} by Elmendorf, K\v rí\v z, Mandell, and May \cite{Elmendorf}. The box product can then be shown to restrict to a symmetric monoidal structure on $*$-modules, and similarly for $S$-modules.
	\end{rem}

	In analogy with the topological story we therefore define:

	\begin{defn}\label{def:starMod}
		An $E\mathcal M$-simplicial set $X$ is called a \emph{$*$-module} if the above map $\Phi\colon E\Inj(2\times\omega,\omega)\times_{E\mathcal M^2}(X\times *)\to X\times *\cong X$ is an isomorphism.
	\end{defn}

	As the main results of this section we will now compare this to the notions studied in the previous subsection. Namely we will show:

	\begin{thm}\label{thm:starModEq}
		An $E\mathcal M$-simplicial set $X$ is a $*$-module if and only if it is mild.
	\end{thm}

	\begin{thm}\label{thm:boxprodOperadic}
		Let $X_1,\dots,X_n$ be mild $E\mathcal M$-simplicial sets. Then $\Phi$ restricts to an isomorphism $E\Inj(n\times\omega,\omega)\times_{E\mathcal M^n} X_1\times\cdots X_n\cong X_1\boxtimes\cdots\boxtimes X_n$.
	\end{thm}

	The proof of the second theorem will require some preparations; for now let us already use it to prove the first one:

 	\begin{proof}[Proof of Theorem~\ref{thm:starModEq}]
 		If $ X $ is mild, then Theorem \ref{thm:boxprodOperadic} in particular implies that $\Phi: E{\Inj}(2\times \omega, \omega) \times_{{\EM}^2} (X\times *) \to X\boxtimes *$ is an isomorphism. But $X\boxtimes *=X\times *$, whence $ X $ is a $ * $-module.

 		Conversely, assume $ X $ is a $ * $-module and let $ x\in X_n $ arbitrary.	Then there are $ f_0, \dots f_n \in \Inj(2\times \omega, \omega) $ and $ y \in X_n $ with
 		$x = \Phi[f_0, \dots , f_n;y] = (f_0 \iota_1, \dots, f_n \iota_1).y$ by assumption. Lemma~\ref{lem:ksupp} therefore shows that $ x $ is $ k $-supported on $ \im(f_k\iota_1) $ for any $k$, which is co-infinite as its complement contains $\im(f_k\iota_2)$.
 	\end{proof}

	We now turn to the preparations for the proof of Theorem~\ref{thm:boxprodOperadic}. We begin with the following converse to Lemma~\ref{lem:ksupp}:

	\begin{lem}\label{lem:kfsupp}
		Let $X$ be an $E\mathcal M$-simplicial set, let $x\in X_n$ be co-infinitely supported, let $u_0,\dots,u_n\in\mathcal M$, and let $A\subset\omega$ co-infinite such that $(u_0,\dots,u_n).x$ is $k$-supported on $u_k(A)$. Then $x$ is $k$-supported on $A$.
		\begin{proof}
			Replacing $X$ by $X^\mu$ (which contains $x$ by assumption), we may assume without loss of generality that $X$ is mild.

			Let now $B$ be co-infinite such that $x$ is $k$-supported on $B$. By Lemma~\ref{lem:ksupp}, $(u_0,\dots,u_n).x$ is $k$-supported on $u_k(B)$, hence on $u_k(A\cap B)=u_k(A)\cap u_k(B)$ by Proposition~\ref{prop:capSupp}. We may therefore assume without loss of generality that $A\subset B$, so it will be enough by Lemma~\ref{lem:agreeSupp}-(\ref{item:aS-preimage}) applied to $i_k^*X_n$ to show that $i_k(u_k).x$ is $k$-supported on $u_k(A)$.

			For this we let $f$ be any injection fixing $u_k(A)$ pointwise. Then
			\begin{align*}
				&(u_0,\dots,u_{k-1},1,u_{k+1},\dots,u_n).i_k(f).i_k(u_k).x\\
			 	&\qquad= i_k(f).(u_0,\dots,u_n).x = (u_0,\dots,u_n).x\\
				&\qquad= (u_0,\dots,u_{k-1},1,u_{k+1},\dots,u_n).i_k(u_k).x
			\end{align*}
			by assumption, hence $i_k(f).i_k(u_k).x=i_k(u_k).x$ since $(u_0,\dots,u_{k-1},1,u_{k+1},\dots,u_n)$ acts injectively by Lemma~\ref{lem:injAct}. Letting $f$ vary, the claim follows immediately.
		\end{proof}
	\end{lem}

 	\begin{lem}\label{lem:Inj/M}
 		Let $ n\geq 1 $ and let $ A_1,\dots,A_n $ be subsets of $ \omega $ with infinite complement. Consider the  equivalence relation on $ \Inj (n\times \omega,\omega) $ generated by the relation $\phi \sim \phi(f_1 \amalg \dots \amalg f_n)$ for all $ f_i\in \M_{A_i} $, $i=1,\dots,n$. Then two elements of $ \Inj (n\times \omega,\omega) $ are equivalent if and only if they agree on the subsets $ \{i\}\times A_i $ of $ n\times \omega $ for all $ i=1,\dots,n. $
 	\end{lem}
 	\begin{proof}
 		The `only if' part is clear. For the `if' part, we fix $f,g\in\Inj(n\times\omega,\omega)$ agreeing on $\bigcup_{i=1}^n\{i\}\times A_i$, and we pick bijections $\phi_i\colon\omega\setminus A_i\cong\omega$ and $\psi\colon \omega\setminus f(\bigcup_{i=1}^n\{i\}\times A_i)\cong\omega$. Given $h\in\Inj(n\times\omega,\omega)$, we then define $\tilde h\colon n\times\omega\to\omega$ via
		\begin{equation*}
			\tilde h(i,x)=\begin{cases}
				f(i,x)=g(i,x) & \text{if }x\in A_i\\
				\psi h(i,\phi_i(x)) & \text{otherwise}.
			\end{cases}
		\end{equation*}
		One then easily checks that $h\mapsto \tilde h$ defines a map $\Inj(n\times\omega,\omega)\to\Inj(n\times\omega,\omega)$ hitting $f$ and $g$, and that this descends to $\Inj(n\times\omega,\omega)/\mathcal M^n\to\Inj(n\times\omega,\omega)/{\sim}$. However, the left hand side consists of a single equivalence class by \cite[Lemma A.5]{schwede}, so necessarily $[f]=[g]$ and the claim follows.
 	\end{proof}

 	\begin{lem}
 		\label{lem:f.xEqualg.x}
 		Let $ X_1,\dots,X_n\in \EMSSetm $, $ (x_1,\dots,x_n)\in (X_1,\dots,X_n)_m $ with $x_j$ $k$-supported on some co-infinite $ A_{k}^{(j)} $ for all $j,k$, and let $ f_0,\dots,f_m,g_0,\dots,g_m \in \Inj(n\times\omega,\omega) $ such that $ f_k $ and $ g_k $ agree on $\bigcup_{j=1}^n \{j\}\times A_{k}^{(j)}  $ for all $k$. Then $$ [f_0, \dots , f_m;x_1,\dots,x_n] = [g_0, \dots , g_m;x_1,\dots,x_n] $$ in $ E{\Inj}(n\times \omega, \omega) \times_{{\EM}^n} (X_1\times\dots\times X_n) $.
 	\end{lem}
 	\begin{proof}
		This follows inductively from the previous lemma by the same argument as in the tame case \cite[Lemma 2.1.20]{lenzGglobal}.
 	\end{proof}

 	\begin{proof}[Proof of Theorem~\ref{thm:boxprodOperadic}]
 		With all of the above results at hand, we can employ a similar strategy to the tame case considered in \cite[Theorem~2.1.19]{lenzGglobal}.

		Lemma \ref{lem:ksupp} implies that the image of $ \Phi $ is contained in $ X_1\boxtimes \dots \boxtimes X_n $.
 		For the other inclusion, let $ (x_1,\dots,x_n)\in (X_1\boxtimes\dots\boxtimes X_n)_m $ and pick for each $0\le k\le m$ pairwise disjoint $A_{k}^{(1)},\dots,A_k^{(n)}$ with $\bigcup_{j=1}^n A^{(j)}_k$ co-infinite such that $x_j$ is $k$-supported on $A_{k}^{(j)}$ (see Remark~\ref{rem:unbiasedBox}). Then we can choose $ f_k\in \Inj(n\times \omega,\omega) $ with $ f_k(j,t) = t $ for all $ t\in A_{k}^{(j)} $. As $ f_k\iota_j\in \M_{ A_{k}^{(j)} } $, we then get
 		\begin{align*}
 		\Phi[f_0,\dots, f_m;x_1,\dots,x_n] &= ((f_0 \iota_1, \dots, f_m \iota_1).x_1,\dots,(f_0 \iota_n, \dots, f_m \iota_n).x_n) \\
 		&= (i_0(f_0 \iota_1) \dots i_n(f_m \iota_1).x_1,\dots, i_0(f_0 \iota_n) \dots i_n(f_m \iota_n).x_n) \\
 		&= (x_1,\dots,x_n).
 		\end{align*}

 		For injectivity, let
		\begin{align*}
		f_0,\dots,f_m, g_0,\dots,g_m &\in \Inj(n\times\omega,\omega)\\
		(x_1,\dots,x_n),(y_1,\dots,y_n)&\in (X_1\times\dots\times X_n)_m
		\end{align*}
		such that $\Phi[f_0,\dots,f_m;x_1,\dots,x_n] \allowbreak=\allowbreak \Phi[g_0,\dots,g_m;y_1,\dots,y_n]$ and we want to show that $[f_0,\dots,f_m;x_1,\dots,x_n] = [g_0,\dots,g_m;y_1,\dots,y_n]$ in $ E{\Inj}(n\times \omega, \omega) \times_{{\EM}^n} (X_1\times\dots\times X_n) $.

 		For this, we will first consider the special case that $ f_i = g_i $ for $ 0\leq i\leq m $. Then $(f_0\iota_j,\dots,f_m\iota_j).x_j=(f_0\iota_j,\dots,i_m\iota_j).y_j$ for any $j$ by definition of $\Phi$, hence also $x_j=y_j$ by Lemma~\ref{lem:injAct} as claimed.

 		In the general case we pick for every $ 0\leq k\leq m $ and $1\le j\le n$ co-infinite sets $ {A'}_{k}^{(j)}, {B'}_{k}^{(j)}\subset \omega $ such that $ x_j $ is $ k $-supported on $  {A'}_{k}^{(j)} $ and $ y_j $ is $ k $-supported on $ {B'}_{k}^{(j)} $. By Lemma \ref{lem:ksupp}, $ (f_0 \iota_j, \dots, f_m \iota_j).x_j = (g_0 \iota_j, \dots, g_m \iota_j).y_j  $ is $ k $-supported on $ f_k\iota_j( {A'}_{k}^{(j)}) $ and on $ g_k\iota_j({B'}_{k}^{(j)}) $, whence on $C_k^{(j)}\mathrel{:=} f_k\iota_j( {A'}_{k}^{(j)})\cap g_k\iota_j({B'}_{k}^{(j)})$ by Proposition~\ref{prop:capSupp}. Lemma~\ref{lem:kfsupp} then shows that $x_j$ is $k$-supported on $A_k^{(j)}\mathrel{:=} (f_k\iota_j)^{-1}(C_k^{(j)})$ while $y_j$ is $k$-supported on $B_k^{(j)}\mathrel{:=} (g_k\iota_j)^{-1}(C_k^{(j)})$. Altogether, we have therefore achieved that $x_j$ and $y_j$ are $k$-supported on co-infinite sets $A_k^{(j)}$ and $B_k^{(j)}$, respectively, such that  $(f_k\iota_j)(A_k^{(j)})=(g_k\iota_j)(B_k^{(j)})$.

		By injectivity of $ f_k\iota_j $ and $ g_k\iota_j $, there is then a (unique) bijection $\sigma_k^{(j)}\mskip-.5\thinmuskip\colon A_k^{(j)}\to B_k^{(j)}$ with $g_k\iota_j\sigma^{(j)}_k=f_k\iota_j|_{A_k^{(j)}}$, which we can further extend to an injection $s_k^{(j)}\in\mathcal M$ as $B_k^{(j)}$ is co-infinite. But then $g_k(s_k^{(1)}\amalg\cdots\amalg s_k^{(n)})$ agrees with $f_k$ on $\bigcup_{j=1}^n\{j\}\times A_k^{(j)}$, so Lemma~\ref{lem:f.xEqualg.x} shows
		\begin{equation}\label{eq:mthm-meq}
			\begin{aligned}
			&[f_0,\dots,f_m;x_1,\dots,x_n]\\
			&\qquad= [g_0(s_0^{(1)}\amalg\cdots\amalg s_0^{(n)}),\dots,g_m(s_m^{(1)}\amalg\cdots\amalg s_m^{(n)});x_1,\dots,x_n]\\
			&\qquad= [g_0,\dots,g_m;(s_0^{(1)},\dots,s_m^{(1)}).x_1,\dots,(s_0^{(n)},\dots,s_m^{(n)}).x_n];
			\end{aligned}
		\end{equation}
 		in particular their images under $\Phi$ agree, whence
 		\begin{align*}
 		\Phi[g_0,\dots,g_m;y_1,\dots,y_n]
 		&=\Phi[f_0,\dots,f_m;x_1,\dots,x_n]\\
 		&= \Phi[g_0,\dots,g_m;(s_0^{(1)},\dots,s_m^{(1)}).x_1,\dots,(s_0^{(n)},\dots,s_m^{(n)}).x_n].
 		\end{align*}
 		The special case above then shows that $ (s_0^{(j)},\dots,s_m^{(j)}).x_j=y_j $, and plugging this into $(\ref{eq:mthm-meq})$ finishes the proof.
 	\end{proof}

 	As a consequence of the theorem, we can also express the right adjoint $ (-)^{\mu} $ in terms of the operadic product:

 	\begin{prop}\label{prop:mildFunOperadic}
 		Let $ X $ be any $ \EM $-simplicial set. Then $ \Phi $ restricts to an isomorphism $ E{\Inj}(2\times \omega, \omega) \times_{{\EM}^2} (X\times *) \to X^{\mu} \times * \cong X^{\mu} $.
 	\end{prop}
 	\begin{proof}
		Consider the naturality square
 		\begin{equation*}
 		\begin{tikzcd}
 		{E{\Inj}(2\times \omega, \omega) \times_{{\EM}^2} (X^{\mu}\times *)} \arrow[r] \arrow[d,"\Phi"',"\cong"] & {E{\Inj}(2\times \omega, \omega) \times_{{\EM}^2} (X\times *)} \arrow[d,"\Phi",]\\
 		{X^{\mu}} \arrow[r,hook] & {X}
 		\end{tikzcd}
 		\end{equation*}
		in which the left hand vertical map is an isomorphism by Theorem~\ref{thm:boxprodOperadic} while the lower horizontal map is injective by definition. It follows formally that also the top horizontal map is injective. We will now prove that it is also surjective, whence an isomorphism, which will then immediately imply the claim.

 		For this we take $ f_0,\dots,f_m\in {\Inj}(2\times \omega, \omega) $ and $ x\in X_m $ arbitrary. If we now let $d\in\mathcal M^\mu$ be the map with $d(x)=2x$, then
		\begin{equation*}
			[f_0(\id\amalg d),\dots,f_m(\id\amalg d);x,*]=[f_0,\dots,f_m;x,(d,\dots,d).*]=[f_0,\dots,f_m;x,*]
		\end{equation*}
		by definition of the operadic product; up to replacing $f_j$ by $f_j(\id\amalg d)$ we may therefore assume that each $f_j$ misses an infinite set $A_j\subset\omega$. We now pick for each $j$ an injection $\phi_j\colon \{1,3,5,\dots\}\to A_j$ and we define $g_j\colon 2\times\omega\to\omega$ via
		\begin{equation*}
			g_j(i,x)=\begin{cases}
				f_j(i,x/2) & \text{if $i=1$ and $x$ is even}\\
				\phi_j(x) & \text{if $i=1$ and $x$ is odd}\\
				f_j(i,x) & \text{otherwise}.
			\end{cases}
		\end{equation*}
		Then $g_j$ is again injective and $g_j(d\amalg\id)=f_j$, so that
		\begin{equation*}
			[f_0,\dots,f_m;x,*]=[g_0,\dots,g_m;(d,\dots,d).x,*].
		\end{equation*}
		But by Lemma~\ref{lem:ksupp} $(d,\dots,d).x$ is $k$-supported on the co-infinite set $\im(d)$ for all $k$, i.e.~$(d,\dots,d).x\in X^\mu$, finishing the proof.
 	\end{proof}

	\subsection{Further categorical properties} The operadic product clearly preserves colimits in each variable separately. Theorem~\ref{thm:boxprodOperadic} therefore implies:

	\begin{cor}\label{cor:boxprod-closed}
		The box product on $\EMSSetm$ preserves colimits in each variable separately.\qed
	\end{cor}

 	On the other hand, we can now prove:

	\begin{thm}\label{thm:mild-presentable}
		The functor $(-)^\mu\colon\EMSSet\to\EMSSetm$ admits a right adjoint, exhibiting $\EMSSetm$ as an accessible Bousfield localization of $\EMSSet$. In particular, $\EMSSetm$ is locally presentable.
		\begin{proof}
			By Proposition~\ref{prop:mildFunOperadic} the endofunctor $(-)^\mu$ of $\EMSSet$ is isomorphic to $E\Inj(2\times\omega,\omega)\times_{E\mathcal M^2}({-}\times *)$, hence in particular cocontinuous. As the enriched functor category $\EMSSet$ is locally presentable, the Special Adjoint Functor Theorem \cite[Corollary~5.5.2.9]{lurie} therefore implies that it admits a right adjoint $R$, which we can restrict to $\EMSSetm$ to obtain the desired right adjoint of $(-)^\mu\colon\EMSSet\to\EMSSetm$. Note that this is automatically fully faithful, as the \emph{left} adjoint of $(-)^\mu$ is simply the inclusion, hence fully faithful.

			It only remains to show that the composite $R(-)^\mu\colon \EMSSet\to\EMSSet$ is accessible. However, $(-)^\mu$ is itself again a right adjoint, whence so is $R(-)^\mu$. As right adjoint functors of locally presentable categories are accessible \cite[Proposition~5.4.7.7]{lurie}, the claim follows.
		\end{proof}
	\end{thm}

	\begin{rem}
		The analogue of the previous theorem in the tame world does \emph{not} hold: the right adjoint $(-)^\tau\colon\EMSSet\to\EMSSett$ is not cocontinuous. As a simple counterexample, the right $\mathcal M$-action on $E\mathcal M$ via precomposition defines a diagram $F\colon B\mathcal M^{\op}\to \EMSSet$ with colimit the terminal $E\mathcal M$-simplicial set (apply Proposition~\ref{prop:equalInM/M_A} in the special case $A=\emptyset$), hence also $(\colim F)^\tau=*$. However, $E\mathcal M^\tau=\emptyset$ (Examples~\ref{ex:mildM} and~\ref{ex:EX-supp}), hence also $\colim \big((-)^\tau\circ F\big)=\emptyset$.
	\end{rem}

 	\section{Commutative \texorpdfstring{$*$}{*}-algebras and their homotopy theory}\label{sec:*-alg}
	\subsection{A reminder on global homotopy theory}
	We will now recall how $E\mathcal M$-simplicial sets can be used to model \emph{global homotopy theory} with respect to finite groups in the sense of \cite{schwede-book}, and more generally how $E\mathcal M$-$G$-simplicial sets for finite groups $G$ (i.e.~simplicial sets with an action of $E\mathcal M\times G$, or equivalently $G$-objects in $\cat{$\bm{E\mathcal M}$-SSet}$) model \emph{$G$-global homotopy theory} in the sense of \cite{lenzGglobal}. For this we will need the following terminology:

	\begin{defn}
		Let $H$ be a finite group. A \emph{complete $H$-set universe} is a countable $H$-set into which any other countable $H$-sets embeds equivariantly. A finite subgroup $H\subset\mathcal M$ is called \emph{universal} if $\omega$ with the restriction of the tautological $\mathcal M$-action is a complete $H$-set universe.
	\end{defn}

	Universal subgroups of $\mathcal M$ abound: any finite group $H$ admits an injective homomorphism $H\to\mathcal M$ with universal image, and this injection is unique up to conjugating with an invertible element of $\mathcal M$, see \cite[Lemma~1.2.8]{lenzGglobal}.

	\begin{defn}
		Let $G$ be a finite group, let $H\subset\mathcal M$ be a subgroup, and let $\phi\colon H\to G$ be any homomorphism. If $X$ is an $(E\mathcal M\times G)$-simplicial set, then we write $X^\phi$ for the fixed points with respect to the \emph{graph subgroup} $\Gamma_{H,\phi}=\{(h,\phi(h)) :h\in H\}$.

		We call an $(E\mathcal M\times G)$-equivariant map $f\colon X\to Y$ a \emph{$G$-global weak equivalence} if the induced map $f^\phi\colon X^\phi\to Y^\phi$ is a weak homotopy equivalence for every \emph{universal} $H\subset\mathcal M$ and every homomorphism $\phi\colon H\to G$.
	\end{defn}

	\begin{prop}\label{prop:g-gl-model}
		Let $G$ be any finite group. Then the category $\cat{$\bm{E\mathcal M}$-$\bm G$-SSet}$ of $(E\mathcal M\times G)$-simplicial sets and equivariant maps admits a combinatorial, simplicial, proper model structure with weak equivalences the $G$-global weak equivalences and generating cofibrations
		\begin{equation*}
			\{E\mathcal M\times_\phi G\times(\del\Delta^n\hookrightarrow\Delta^n) : \text{$H\subset\mathcal M$ universal, $\phi\colon H\to G$, $n\ge0$}\},
		\end{equation*}
		where `$\times_\phi$' denotes the quotient of the product by the diagonal of the evident right $H$-action on $E\mathcal M$ and the right $H$-action on $G$ via $\phi$. 	We call this model structure the \emph{$G$-global model structure}.

		Moreover, $\cat{$\bm{E\mathcal M}$-$\bm G$-SSet}$ also admits a \emph{$G$-global injective model structure} with the same weak equivalences and with cofibrations those maps that are underlying cofibrations of simplicial sets. This model structure is again combinatorial, simplicial, and proper, and the identity constitutes a Quillen equivalence
		\begin{equation}\label{eq:id-QE}
			\cat{$\bm{E\mathcal M}$-$\bm G$-SSet}_\textup{$G$-gl}\rightleftarrows\cat{$\bm{E\mathcal M}$-$\bm G$-SSet}_\textup{inj.~$G$-gl}.
		\end{equation}
		Finally, the $G$-global weak equivalences are stable under filtered colimits.
	\end{prop}
	\begin{proof}
		These model structures are special cases of \cite[Proposition~1.1.2]{lenzGglobal}, also see Corollary~1.2.34 of \emph{op.~cit.} for the former model structure. Finally, the identity defines a Quillen equivalence $(\ref{eq:id-QE})$ as both model structures have the same weak equivalences and any (generating) cofibration of the $G$-global model structure is an injective cofibration.
	\end{proof}

	As already promised above, for $G=1$ this specializes to a model of global homotopy theory for finite groups: there is a zig-zag of Quillen equivalences to Schwede's $\mathcal F\mskip-.5\thinmuskip in$-global model structure on orthogonal spaces \cite[Theorem~1.4.8]{schwede-book}, see \cite[Theorems~1.4.30 and~1.4.31, Corollary~1.5.29]{lenzGglobal}. On the other hand, for general $G$ the above refines $G$-equivariant homotopy theory in the following sense:

	\begin{defn}
		A map $f$ in $\cat{$\bm{E\mathcal M}$-$\bm G$-SSet}$ is called a \emph{$G$-equivariant weak equivalence} if $f^\phi$ is a weak homotopy equivalence for every universal $H\subset\mathcal M$ and every \emph{injective} $\phi\colon H\to G$.
	\end{defn}

	\begin{prop}\label{prop:g-gl-vs-g-equiv}
		Fix an embedding $j\colon G\to\mathcal M$ with universal image and write $\delta\mathrel{:=}(j,\id)\colon G\to\mathcal M\times G$. Then $\delta^*\colon\cat{$\bm{E\mathcal M}$-$\bm G$-SSet}_\textup{$G$-gl}\to\cat{$\bm{G}$-SSet}$ is left Quillen for the usual $G$-equivariant model structure on the target, and the induced adjunction $$\Ho(\cat{$\bm{E\mathcal M}$-$\bm G$-SSet})\rightleftarrows\Ho(\cat{$\bm G$-SSet})$$ on homotopy categories is a Bousfield localization at the $G$-equivariant weak equivalences.
	\end{prop}
	\begin{proof}
		See \cite[Theorem~1.2.92 and Proposition~1.2.95]{lenzGglobal}.
	\end{proof}

	\begin{rem}
		In fact, \emph{loc.~cit.} shows that $\delta^*\colon \Ho(\cat{$\bm{E\mathcal M}$-$\bm G$-SSet})\to\Ho(\cat{$\bm G$-SSet})$ sits in a sequence of four adjoints $\delta^!\dashv\delta_!\dashv\delta^*\dashv\delta_*$, also cf.~\cite{rezkCohesion}.
	\end{rem}

	Finally, the same homotopy theory can also be modelled by tame actions:

	\begin{prop}\label{prop:tame-model}
		The{\hskip0pt minus 1pt} category{\hskip0pt minus 1pt} $\cat{$\bm{E\mathcal M}$-$\bm G$-SSet}^\tau${\hskip0pt minus 1pt} admits{\hskip0pt minus 1pt} a{\hskip0pt minus 1pt} combinatorial{\hskip0pt minus 1pt} model{\hskip0pt minus 1pt} structure with weak equivalences the $G$-global weak equivalences and with generating cofibrations given by the maps
		\begin{equation}\label{eq:gen-cof-tame}
			E\Inj(A,\omega)\times_\phi  G\times(\del\Delta^n\hookrightarrow\Delta^n)
		\end{equation}
		for all $n\ge0$, finite groups $H$, finite non-empty faithful $H$-sets $A$, and homomorphisms $\phi\colon H\to G$. We call this the \emph{positive $G$-global model structure}. It is simplicial, proper, and filtered colimits in it are homotopical.
	\end{prop}

	\begin{thm}\label{thm:tame-vs-all}
		The adjunction $\incl\hskip-1pt\colon\hskip0pt minus 3pt\cat{$\bm{E\hskip0pt minus 1pt\mathcal M}$-$\bm G$-SSet}^\tau\hskip0pt minus 3pt\rightleftarrows\hskip0pt minus 3pt\cat{$\bm{E\hskip0pt minus 1pt\mathcal M}$-$\bm G$-SSet}_\textup{inj.~$\mskip-.5\thinmuskip G$-gl}\hskip0pt minus 3pt:\!(-)^\tau$ is a Quillen equivalence.
	\end{thm}
	\begin{proof}[Proof of Proposition~\ref{prop:tame-model} and Theorem~\ref{thm:tame-vs-all}]
		See \cite[Theorem~1.4.60]{lenzGglobal}.
	\end{proof}

	\begin{rem}
		One can drop the assumption that the set $A$ in $(\ref{eq:gen-cof-tame})$ be non-empty, leading to a \emph{$G$-global model structure} \cite[Remark~1.4.61]{lenzGglobal}. The above model structure will however be more useful for the study of commutative monoids later.
	\end{rem}

	\subsection{The model structure on simplicial \texorpdfstring{$\bm*$}{*}-modules}
	Throughout, let $G$ be a finite group. In this subsection, we will define similar model structures on $\cat{$\bm{E\mathcal M}$-$\bm G$-SSet}^\mu$ and show that they are equivalent to the above models.

 	\begin{prop}\label{prop:*-mod-quillen}
		The category $\cat{$\bm{E\mathcal M}$-$\bm G$-SSet}^\mu$ admits a combinatorial model structure with weak equivalences the $G$-global weak equivalences and generating cofibrations given by the maps
		\begin{equation}\label{eq:gen-cof-pos-mild}
			E\Inj(A,\omega)^\mu\times_\phi G\times(\del\Delta^n\hookrightarrow\Delta^n)
		\end{equation}
		for all $n\ge0$, finite groups $H$, non-empty (finitely or infinitely) countable faithful $H$-sets $A$, and homomorphisms $\phi\colon H\to G$. 	We call this the \emph{positive $G$-global model structure}. It is simplicial, proper, and filtered colimits in it are homotopical.
	\end{prop}

	Our construction of this model structure will rely on the following general criterion from \cite[Proposition~A.2.6.13]{lurie}:

	\begin{prop}\label{prop:lurie-criterion}
		Let $\mathscr C$ be a locally presentable category, let $I$ be a \emph{set} of maps in $\mathscr C$, and let $\mathscr W\subset\mathscr C$ be a subcategory. Then there exists a left proper combinatorial model structure on $\mathscr C$ with weak equivalences $\mathscr W$ and generating cofibrations given by $I$ provided that the following conditions are satisfied:
		\begin{enumerate}
			\item The subcategory $\mathscr W\subset\mathscr C$ is \emph{perfect} \cite[Definition~A.2.6.10]{lurie}.
			\item If $j$ is a pushout of a map in $I$ along an arbitrary map in $\mathscr C$, then weak equivalences are stable under pushouts along $j$.
			\item If $p$ is a map in $\mathscr C$ with the right lifting property against all maps in $I$, then $p\in\mathscr W$\kern-2.5pt.\qed
		\end{enumerate}
	\end{prop}

	\begin{rem}\label{rk:all-about-perfect}
		Below, we will treat the notion of a perfect class of morphisms as a blackbox. The only facts we will need are the following:
		\begin{enumerate}
			\item If $\mathscr C$ is a left proper combinatorial model category in which filtered colimits are homotopical, then conversely the subcategory $\mathscr W$ of weak equivalences is perfect \cite[Remark~A.2.6.14]{lurie}.
			\item If $F\colon\mathscr C\to\mathscr D$ is a functor of locally presentable categories preserving filtered colimits and $\mathscr W\subset\mathscr D$ is perfect, then so is $F^{-1}(\mathscr W)\subset\mathscr C$ \cite[Corollary~A.2.6.12]{lurie}.
		\end{enumerate}
	\end{rem}

	We will further need the following homotopical information on the box product:

	\begin{prop}\label{prop:operadic-prod-homotopical}
		Let $X_1,\dots,X_n$ be any $E\mathcal M$-$G$-simplicial sets. Then the natural map $\Phi\colon E\Inj(n\times\omega,\omega)\times_{E\mathcal M^n}X_1\times\cdots\times X_n\to X_1\times\cdots\times X_n$ from Construction~\ref{const:operadBox} is a $G$-global weak equivalence.

		Moreover, if $X_1=X_2=\cdots=X_n$ and we equip both sides with the $\Sigma_n$-action via permuting the factors and acting on $n$ in the tautological way, then $\Phi$ is even a $(G\times\Sigma_n)$-global weak equivalence.
		\begin{proof}
			See \cite[Proposition~5.17]{lenzOperads}.
		\end{proof}
	\end{prop}

	Combining the first half with Proposition~\ref{prop:mildFunOperadic} we immediately get:

	\begin{cor}\label{cor:mu-we}
		Let $X$ be any $E\mathcal M$-$G$-simplicial set. Then the inclusion $X^\mu\hookrightarrow X$ is a $G$-global weak equivalence.\qed
	\end{cor}

	\begin{proof}[Proof of Proposition~\ref{prop:*-mod-quillen}]
		We will verify the assumptions of Proposition~\ref{prop:lurie-criterion}.

		By Theorem~\ref{thm:mild-presentable}, $\cat{$\bm{E\mathcal M}$-$\bm G$-SSet}^\mu$ is locally presentable. Moreover, by the first half of Remark~\ref{rk:all-about-perfect} together with Proposition~\ref{prop:g-gl-model}, the $G$-global weak equivalences in $\cat{$\bm{E\mathcal M}$-$\bm G$-SSet}$ are perfect, hence so are the $G$-global weak equivalences in $\cat{$\bm{E\mathcal M}$-$\bm G$-SSet}^\mu$ by the second half of the remark applied to the inclusion.

		Clearly, all the putative generating cofibrations are injective cofibrations, whence so is any pushout $j$ of any of them. Thus, the second assumption follows from left properness of the injective $G$-global model structure on $\cat{$\bm{E\mathcal M}$-$\bm G$-SSet}$.

		Finally, if a map $p$ enjoys the right lifting property against all maps of the form $E\mathcal M^\mu\times_\phi G\times(\del\Delta^n\hookrightarrow\Delta^n)$ for a fixed universal $H\subset\mathcal M$ and homomorphism $\phi\colon H\to G$, then an easy adjointness argument shows that $\maps^{E\mathcal M}(E\mathcal M^\mu,X)^\phi$ is an acyclic Kan fibration; here $H$ acts via its right action on $\mathcal M^\mu$ and its action on $X$, while $G$ acts on $X$ as before. The inclusion $E\mathcal M^\mu\hookrightarrow E\mathcal M$ has a left-$E\mathcal M$-equivariant-right-$H$-equivariant homotopy inverse as follows: as $H$ is universal, we find an $H$-equivariant injecion $\omega\amalg\omega\to\omega$; restricting to the first summand then gives an $H$-equivariant injection $u\colon\omega\to\omega$ missing infinitely many elements, and right multiplication with $u$ then induces the desired homotopy inverse, with homotopy similarly given by acting with $(1,u)$. Thus, the acyclic Kan fibration $\maps^{E\mathcal M}(E\mathcal M^\mu,X)^\phi$ agrees up to weak equivalence with $p^\phi$; in particular the latter is a weak homotopy equivalence. Letting $\phi$ vary, it follows that $p$ is a $G$-global weak equivalence. This completes the proof that the model structure exists, that it is left proper and combinatorial, and that filtered colimits in it are homotopical.

		Next, we will show that this model structure is simplicial, for which we have to verify the pushout product axiom. The pushout product axiom for (generating) cofibrations is clear. For the acyclicity part, let $i\colon A\to B$ be a cofibration in $\cat{SSet}$ and $f\colon X\to Y$ a cofibration in $\cat{$\bm{E\mathcal M}$-$\bm G$-SSet}^\mu$; we consider the diagram
		\begin{equation*}
			\begin{tikzcd}
				A\times X\arrow[dr,phantom,"\ulcorner"{very near end}]\arrow[d, "i\times X"']\arrow[r, "A\times f"] & A\times Y\arrow[d]\arrow[ddr, bend left=15pt, "i\times Y"]\\
				B\times X\arrow[r]\arrow[drr, bend right=15pt, "B\times f"'] & P\arrow[dr, dashed, "i\ppo f"{description}]\\
				&& B\times Y
			\end{tikzcd}
		\end{equation*}
		defining the pushout product $i\ppo f$. If $f$ is a $G$-global weak equivalence, then so are $A\times f$ and $B\times f$; on the other hand, also the map $B\times X\to P$ is a $G$-global weak equivalence as a pushout of $A\times f$ along the injective cofibration $i\times X$. Thus, $2$-out-of-$3$ shows that $i\ppo f$ is a weak equivalence. The argument for the case where $i$ is a weak homotopy equivalence is analogous.

		Finally, we have to show right properness, for which we consider any pullback
		\begin{equation}\label{diag:pb-mild-EM-G}
			\begin{tikzcd}
				A\arrow[d]\arrow[r, "g"]\arrow[dr,phantom,"\lrcorner"{very near start}] & B\arrow[d, "p"]\\
				C\arrow[r, "f"'] & D
			\end{tikzcd}
		\end{equation}
		in $\cat{$\bm{E\mathcal M}$-$\bm G$-SSet}^\mu$ such that $f$ is a $G$-global weak equivalence and $p$ is a fibration in $\cat{$\bm{E\mathcal M}$-$\bm G$-SSet}$; we have to show that $g$ is a $G$-global weak equivalence again.

		For this we fix a universal $H\subset\mathcal M$ together with a homomorphism $\phi\colon H\to G$, and we write $R_\phi\colon \cat{$\bm{E\mathcal M}$-$\bm G$-SSet}^\mu\to\cat{SSet}$ for the functor $\maps^{E\mathcal M}(E\mathcal M^\mu,-)^\phi$. If $f$ is any map of mild $E\mathcal M$-$G$-simplicial sets, then we have seen above that $R_\phi f$ is a weak homotopy equivalence if and only if $f^\phi$ is so.
		On the other hand, $R_\phi$ is right adjoint to the functor $\cat{SSet}\to\cat{$\bm{E\mathcal M}$-$\bm G$-SSet}^\mu,\allowbreak X\mapsto E\mathcal M^\mu\times_\phi G\times X$, and the latter sends generating cofibrations to generating cofibrations and is homotopical. Thus, $R_\phi$ in particular sends fibrations of the positive $G$-global model structure to Kan fibrations. Thus, applying the right adjoint $R_\phi$ to $(\ref{diag:pb-mild-EM-G})$ expresses $R_\phi g$ as the pullback of the weak homotopy equivalence $R_\phi f$ along the Kan fibration $R_\phi p$, so $R_\phi g$ is a weak homotopy equivalence by right properness of the Kan-Quillen model structure, and accordingly also $g^\phi$ is a weak homotopy equivalence as desired.
	\end{proof}

	\begin{rem}
		The argument above works for a wide class of sets of generating cofibrations $I$: namely, we only need that $I$ consists of injective cofibrations, that all $E\mathcal M^\mu\times_\phi G$ are $I$-cofibrant, and that the pushout product of any map in $I$ with $\del\Delta^n\hookrightarrow\Delta^n$ ($n\ge0$) is an $I$-cofibration. In particular, one can use this to obtain a \emph{projective $G$-global model structure} on $\cat{$\bm{E\mathcal M}$-$\bm G$-SSet}^\mu$ with generating cofibrations given by the maps $E\mathcal M^\mu\times_\phi G\times(\del\Delta^n\hookrightarrow\Delta^n)$ as well as a \emph{$G$-global model structure} where we drop the assumption that the set $A$ in $(\ref{eq:gen-cof-pos-mild})$ be non-empty. Finally, if $I$ is a set of generating cofibrations for the injective model structure on $\cat{$\bm{E\mathcal M}$-$\bm G$-SSet}$, applying this to $I^\mu=\{i^\mu :i\in I\}$ can be shown to yield an \emph{injective $G$-global model structure} with the $G$-global weak equivalences as weak equivalences and the underlying cofibrations as cofibrations. As the positive $G$-global model structure will suffice for our purposes, we leave the details to the interested reader.
	\end{rem}

	\begin{prop}\label{prop:*-mod-equiv}
		The adjunctions
		\begin{align*}
			(-)^\mu\colon\cat{$\bm{\EM}$-$\bm G$-SSet}_\textup{$G$-gl.} &\rightleftarrows \cat{$\bm{\EM}$-$\bm G$-SSet}^\mu_\textup{pos.~$G$-gl} :\!R\\
			\incl\colon\cat{$\bm{\EM}$-$\bm G$-SSet}^\mu_\textup{pos.~$G$-gl}&\rightleftarrows\cat{$\bm{\EM}$-$\bm G$-SSet}_\textup{inj.~$G$-gl} :\!(-)^\mu\\
			\incl\colon \cat{$\bm{\EM}$-$\bm G$-SSet}^\tau_\textup{pos.~$G$-gl} &\rightleftarrows\cat{$\bm{\EM}$-$\bm G$-SSet}^\mu_\textup{pos.~$G$-gl} :\!(-)^\tau
		\end{align*}
		are Quillen equivalences.
	\end{prop}
	\begin{proof}
		Corollary~\ref{cor:mu-we} implies that $(-)^\mu\colon\cat{$\bm{E\mathcal M}$-$\bm G$-SSet}\to\cat{$\bm{E\mathcal M}$-$\bm G$-SSet}^\mu$ is homotopical and that the induced functor on homotopy categories is quasi-inverse to the functor induced by $\cat{$\bm{E\mathcal M}$-$\bm G$-SSet}^\mu\hookrightarrow \cat{$\bm{E\mathcal M}$-$\bm G$-SSet}$. On the other hand, also $\cat{$\bm{E\mathcal M}$-$\bm G$-SSet}^\tau\hookrightarrow \cat{$\bm{E\mathcal M}$-$\bm G$-SSet}$ induces an equivalence of homotopy categories by Theorem~\ref{thm:tame-vs-all}, whence so does $\cat{$\bm{E\mathcal M}$-$\bm G$-SSet}^\tau\hookrightarrow \cat{$\bm{E\mathcal M}$-$\bm G$-SSet}^\mu$ by $2$-out-of-$3$.

		It therefore only remains to show that the above adjunctions are Quillen adjunctions; however, for each of these the left adjoint is homotopical and sends generating cofibrations to cofibrations by direct inspection.
	\end{proof}

	\begin{prop}
		The positive $G$-global model structure on $\cat{$\bm{E\mathcal M}$-$\bm G$-SSet}^\mu$ admits a Bousfield localization with weak equivalences the $G$-equivariant weak equivalences; we call this the \emph{positive $G$-equivariant model structure}. It is simplicial, combinatorial, proper, and filtered colimits in it are homotopical.

		If $j\colon G\to\mathcal M$ is a universal embedding and $\delta=(j,\id)\colon G\to \mathcal M\times G$, then $$\delta^*\colon\cat{$\bm{E\mathcal M}$-$\bm G$-SSet}^\mu\to\cat{$\bm G$-SSet}$$ is the left half of a Quillen equivalence with respect to this model structure.
		\begin{proof}
			It is clear that $\delta^*\colon \cat{$\bm{E\mathcal M}$-$\bm G$-SSet}^\mu\to \cat{$\bm G$-SSet}$ is a simplicial left adjoint and that it sends generating cofibrations to cofibrations and $G$-global weak equivalences to $G$-equivariant weak equivalences. Thus, \cite[Corollary~A.3.7.10]{lurie} shows that the desired Bousfield localization exists, is simplicial, left proper, and combinatorial, and that $\delta^*$ descends to a left Quillen functor from the positive $G$-equivariant model structure. The $G$-equivariant weak equivalences are obviously stable under filtered colimits, and as any fibration of the positive $G$-equivariant model structure is also a fibration in the positive $G$-global model structure the same argument as in the proof of Proposition~\ref{prop:*-mod-quillen} shows right properness.

			Finally, the induced functor on homotopy categories is an equivalence by the combination of Propositions~\ref{prop:*-mod-equiv} and~\ref{prop:g-gl-vs-g-equiv}.
		\end{proof}
	\end{prop}

 	\subsection{Parsummable simplicial sets and commutative \texorpdfstring{$\bm*$}{*}-algebras} In this final subsection we will lift the above results to the level of commutative monoid objects, which will give new models of equivariant and globally coherently commutative monoids. For this we first recall:

	\begin{defn}
		A \emph{parsummable simplicial set} is a commutative monoid object for the box product on $\cat{$\bm{E\mathcal M}$-SSet}^\tau$. We write $\cat{ParSumSSet}\mathrel{:=}\CMon(\cat{$\bm{E\mathcal M}$-SSet}^\tau)$ for the category of parsummable simplicial sets and monoid homomorphisms, and more generally $\cat{$\bm G$-ParSumSSet}\mathrel{:=}\CMon(\cat{$\bm{E\mathcal M}$-$\bm G$-SSet}^\tau)$.
	\end{defn}

	\begin{thm}
		The category $\cat{$\bm G$-ParSumSSet}$ of $G$-parsummable simplicial sets admits a model structure in which a map is a weak equivalence or fibration if and only if it so in the positive $G$-global model structure on $\cat{$\bm{E\mathcal M}$-$\bm G$-SSet}^\tau$. We call this the \emph{positive $G$-global model structure} again. It is combinatorial, simplicial, proper, and filtered colimits in it are homotopical.
		\begin{proof}
			See \cite[Theorem~2.1.36]{lenzGglobal}.
		\end{proof}
	\end{thm}

	\cite{lenzGglobal} makes the point that $G$-parsummable simplicial sets have the moral right to be called a model of `$G$-globally coherently commutative monoids': they are equivalent to $G$-global versions of Segal's $\Gamma$-spaces (Theorem~2.3.1 of \emph{op.~cit.}), they admit a $G$-global delooping theorem relating them to stable $G$-global homotopy theory (Theorem~3.4.21), for $G=1$ they recover Schwede's ultra-commutative monoids (Corollary~2.3.17), and for general $G$ they refine $G$-equivariantly coherently commutative monoids in the following sense:

	\begin{thm}
		The localization of $\cat{$\bm G$-ParSumSSet}$ at the \emph{$G$-equivariant} weak equivalences is equivalent to the homotopy theory of Shimakawa's $G$-equivariant special $\Gamma$-spaces \cite{shimakawa, ostermayr}.
	\end{thm}
	\begin{proof}
		See \cite[Corollary~2.1.38 and Theorem~2.3.18]{lenzGglobal}.
	\end{proof}

	We can now finally introduce the key concept of this paper:

	\begin{defn}
		A \emph{commutative $*$-algebra} is a commutative monoid object in $(\cat{$\bm{E\mathcal M}$-SSet}^\mu,\boxtimes)$. We write $\cat{$\bm*$-Alg}\mathrel{:=}\CMon(\cat{$\bm{E\mathcal M}$-SSet}^\mu)$ for the category of commutative $*$-algebras and more generally $\cat{$\bm G$-$\bm*$-Alg}\mathrel{:=}\CMon(\cat{$\bm{E\mathcal M}$-$\bm G$-SSet}^\mu)$ for any finite group $G$.
	\end{defn}

	\begin{rem}
		Commutative monoids for the box product $\boxtimes_{\mathcal L}$ on $\mathcal L$-spaces (as recalled in Remark~\ref{rk:boxtimes-classical}) are studied in \cite[Section~4]{blum}; in particular, by Theorem~4.19 of \emph{op.~cit.} they admit a model structure modelling non-equivariant coherently commutative monoids.
	\end{rem}

	\begin{rem}\label{rk:*-alg-vs-operadic-alg}
		The above commutative $*$-algebras are closely related to algebras over the injections operad $\mathcal I$: namely, Theorem~\ref{thm:boxprodOperadic} implies via the same argument as in \cite[Proposition~5.22]{lenzOperads} that we have an \emph{isomorphism of categories} between $\mathcal I$-algebras in $\cat{SSet}$ whose underlying $E\mathcal M$-simplicial set is mild and commutative $*$-algebras given by sending an $\mathcal I$-algebra $\mathcal A$ to its underlying $E\mathcal M$-simplicial set with sum given by the composite $\mathcal A\boxtimes\mathcal A\cong E\Inj(2\times\omega,\omega)\times_{E\mathcal M^2}\mathcal A^{\times 2}\to\mathcal A$ of the inverse of the isomorphism $\Phi$ and the map induced by the action of the $2$-ary operations, also cf.~\cite[V.3]{kriz-may} for a similar statement in the realm of chain complexes or \cite[Proposition~4.8]{blum} for the corresponding result in the world of $\mathcal L$-spaces.
	\end{rem}

	\begin{thm}\label{thm:*-alg-model}
		The category $\cat{$\bm G$-$\bm*$-Alg}$ admits a model structure in which a map is a weak equivalence or fibration if and only if it is so in the positive $G$-global model structure on $\cat{$\bm{E\mathcal M}$-$\bm G$-SSet}^\mu$. We call this the \emph{positive $G$-global model structure} again; it is combinatorial, simplicial, proper, and filtered colimits in it are homotopical. Moreover, the adjunction $\incl\colon\cat{$\bm G$-ParSumSSet}\rightleftarrows\cat{$\bm G$-$\bm*$-Alg}:\!(-)^\tau$ is a Quillen equivalence.
	\end{thm}

	The proof will be given below after some recollections and preparations. Before that, we note the following direct $G$-equivariant consequence:

	\begin{cor}\label{cor:g-*-equivariant}
		The localization of $\cat{$\bm G$-$\bm*$-Alg}$ at the \emph{$G$-equivariant} weak equivalences is equivalent to the homotopy theory of special $\Gamma$-$G$-spaces.\qed
	\end{cor}

	\begin{rem}
		The $G$-equivariant weak equivalences are again part of a suitable model structure on $\cat{$\bm G$-$\bm*$-Alg}$, which can be constructed by exactly the same arguments as we will use for the positive $G$-global model structure below.
	\end{rem}

	Just like its tame sibling, the model structure from Theorem~\ref{thm:*-alg-model} will be obtained via the general machinery of \cite{white}, and we begin by recalling the necessary terminology.

 	\begin{defn}
 		A symmetric monoidal model category $ \mathscr{C} $ satisfies the \emph{monoid axiom} if every transfinite composition of pushouts of maps of the form $ X\otimes j $ with $ X\in \mathscr{C} $ and $ j $ an acyclic cofibration is a weak equivalence.
 	\end{defn}

 	\begin{const}
 		Let $ \mathscr{C} $ be a cocomplete symmetric monoidal category; for ease of notation we will pretend below that its tensor product is strictly associative.

		We let $C_n$ denote the $n$-cube, i.e.~the poset of subsets of $\{1,\dots,n\}$.  Given any map $ f\colon X\to Y$ in $\mathscr C$, we obtain a diagram $K^n(f)\colon C_n\to\mathscr C$ by sending $ I\subset \{1,\dots,n\} $ to $ Z_1\otimes \dots \otimes Z_n $ where $ Z_i = Y $ for $ i\in I $ and $ Z_i = X $ otherwise, with structure maps the appropriate tensor products of $ f $ and the respective identities.

 		We write $ Q^n (f) $ for the colimit of the subdiagram of $ C_n $ where the terminal vertex has been removed, which comes with a map $f^{\ppo n}\colon Q^n(f)\to Y^{\otimes n}$ induced by the remaining structure maps. Moreover, $ Q^n (f) $ has a natural $ \Sigma_n $ action induced by the $ \Sigma_n $-action on $ C_n $ and the symmetry isomorphisms of the tensor product; with respect to this action and the permutation action on $Y^{\otimes n}$, the above map $f^{\ppo n}$ is $\Sigma_n$-equivariant; in particular, it descends to a map $Q^n(f)/\Sigma_n\to Y^{\otimes n}/\Sigma_n$.
 	\end{const}

 	\begin{defn}
 		A symmetric monoidal model category $ \mathscr{C} $ satisfies the \emph{strong commutative monoid axiom} if $ i^{\ppo n} /\Sigma_n $ is a cofibration for all $ n\geq 0 $ and every cofibration $ i$ of $\mathscr{C} $, acyclic whenever $i$ is.
 	\end{defn}

 	\begin{thm}[White]\label{thm:comMon}
 		Let $ \mathscr{C} $ be a combinatorial symmetric monoidal model category that satisfies the monoid axiom and the strong commutative monoid axiom. Then $\CMon(\mathscr{C}) $ admits a model structure in which a map is a weak equivalence or fibration if and only if it is so in $\mathscr C$. This model structure is again combinatorial, and it is right proper if $\mathscr C$ is so. Finally, if $\mathscr C$ is a simplicial model category such that the tensor product preserves the $\cat{SSet}$-tensoring in each variable, then $\CMon(\mathscr C)$ is again simplicial.
 	\end{thm}
 	\begin{proof}
 		The existence of the model structure and its combinatoriality is \cite[Theorem 3.2]{white} and the discussion after it. The remaining statements are standard facts about transferred model structures \cite[Lemma~A.2.14]{lenzGglobal}.
 	\end{proof}

	The question of left properness is somewhat more subtle:

	\begin{prop}\label{prop:cmon-left-proper}
		In the situation of the previous theorem, assume the following:
		\begin{enumerate}
			\item $\mathscr C$ is left proper and filtered colimits in it are homotopical.
			\item $\mathscr C$ admits a set of generating cofibrations with cofibrant sources.
			\item For any $X\in\mathscr C$ and any cofibration $i$, pushouts along $X\otimes i$ are homotopy pushouts.
			\item For any cofibrant $X$, the functor $X\otimes{-}$ is homotopical.
		\end{enumerate}
		Then the above model structure on $\CMon(\mathscr C)$ is again left proper.
		\begin{proof}
			This is an instance of \cite[Theorem~4.17]{white}.
		\end{proof}
	\end{prop}

	In order to apply this to our situation, we first have to understand the homotopical behaviour of the box product better.

	\begin{lem}\label{lem:boxprod-homotopical}
		The box product on $\cat{$\bm{E\mathcal M}$-$\bm G$-SSet}^\mu$ preserves $G$-global weak equivalences in each variable.
		\begin{proof}
			The cartesian product on $\cat{$\bm{E\mathcal M}$-$\bm G$-SSet}$ obviously has this property, so the lemma follows by combining Theorem~\ref{thm:boxprodOperadic} with Proposition~\ref{prop:operadic-prod-homotopical}.
		\end{proof}
	\end{lem}

	\begin{prop}\label{prop:box-powers}
		Let{\hskip0pt minus 1pt} $f\colon X\to Y${\hskip0pt minus 1pt} be{\hskip0pt minus 1pt} a{\hskip0pt minus 1pt} $G$-global{\hskip0pt minus 1pt} weak{\hskip0pt minus 1pt} equivalence{\hskip0pt minus 1pt} in{\hskip0pt minus 1pt} $\cat{$\bm{E\mathcal M}$-$\bm G$-SSet}^\mu$ and let $n\ge0$. Then $f^{\boxtimes n}$ is a $(G\times\Sigma_n)$-global weak equivalence with respect to the $\Sigma_n$-action permuting the factors.

		Moreover, if both $X$ and $Y$ have no vertices supported on the empty set (i.e.~$\mathcal M$-fixed vertices), then $f^{\boxtimes n}/\Sigma_n$ is a $G$-global weak equivalence.
	\end{prop}
	\begin{proof}
		The first statement follows like in the previous lemma, using that $f^{\times n}$ is a $G$-global weak equivalence in $\cat{$\bm{E\mathcal M}$-$\bm G$-SSet}$ by \cite[Corollary~1.2.79]{lenzGglobal}, also see Corollary~1.4.71 of \emph{op.~cit.}

		For the second statement, it will then be enough by \cite[Corollary~1.2.80]{lenzGglobal} that $\Sigma_n$ acts freely on both $X^{\boxtimes n}$ and $Y^{\boxtimes n}$. But indeed, let $Z$ be a mild $\mathcal M$-simplicial set and $(z_1,\dots,z_n)$ any vertex of $Z^{\boxtimes n}$. Then we can pick pairwise disjoint co-infinite sets $A_i$ such that $z_i$ is supported on $A_i$; thus, if $z_i=z_j$ for some $i\not=j$, then both are supported on $A_i\cap A_j=\emptyset$ by Proposition~\ref{prop:capSupp}. In particular, if $Z$ has no vertices supported on the empty set, then the $z_i$'s are pairwise distinct and therefore $\Sigma_n$ acts freely as claimed.
	\end{proof}

	\begin{lem}\label{lem:cofibrant-no-fixed-points}
		Let $X\in\cat{$\bm{E\mathcal M}$-$\bm G$-SSet}^\mu$ be cofibrant in the positive $G$-global model structure. Then no vertex of $X$ is supported on the empty set.
		\begin{proof}
			We will show more generally: \emph{if $i$ is a cofibration in the positive $G$-global model structure, then no vertex outside the image of $i$ is supported on the empty set.}

			For this we first observe that if $A$ is non-empty and $H$ acts arbitrarily on $A$, then no vertex of $E\Inj(A,\omega)/H$ is supported on the empty set by the same argument as in Example~\ref{ex:injA}. In particular, if $\phi\colon H\to G$ is any homomorphism, the same holds true for $E\Inj(A,\omega)\times_\phi G$ as this admits an $E\mathcal M$-equivariant map to $E\Inj(A,\omega)/H$. Thus, the claim holds for the generating cofibrations $(\ref{eq:gen-cof-pos-mild})$.

			To complete the proof, it suffices now to prove that the class of \emph{injective cofibrations} $i$ satisfying the above property is closed under pushouts, transfinite compositions, and retracts. 	The latter two are easy, so we will argue for the first one. As pushouts are computed levelwise, this amounts to saying that given any pushout
			\begin{equation*}
				\begin{tikzcd}
					A\arrow[r, "i"]\arrow[d]\arrow[dr, phantom, "\ulcorner"{very near end}] & B\arrow[d,"g"]\\
					C\arrow[r, "j"'] & D
				\end{tikzcd}
			\end{equation*}
			in $\mathcal M$-sets such that $i$ is injective and every $b\in B^{\mathcal M}$ is contained in the image of $i$, then every $d\in D^\mathcal M$ is contained in the image of $j$. But by the standard construction of pushouts along injections in sets, $j$ and $g$ exhibit $D$ as the disjoint union of $C$ and $B\setminus i(A)$. By Lemma~\ref{lemma:complement}, $B\setminus i(A)$ is closed under the $\mathcal M$-action on $B$, so this is necessarily a decomposition of $\mathcal M$-sets and the claim follows immediately.
		\end{proof}
	\end{lem}

	\begin{lem}\label{lem:more-cofibrations}
		Let $H$ be a finite group, let $A$ be a non-empty countable faithful $H$-set, and let $i\colon X\to Y$ be a levelwise injection of $(G\times H)$-simplicial sets such that $G$ acts freely on $Y$ outside the image of $i$. Then $E\Inj(A,\omega)^\mu\times_H i$ is a cofibration in $\cat{$\bm{E\mathcal M}$-$\bm G$-SSet}^\mu$.
		\begin{proof}
			We recall from \cite[Proposition~2.16]{steph} that any such map $i$ can be written as a (retract of a) relative cell complex with respect to the set
			\begin{equation*}
				\big\{(G\times H)/\Gamma_{K,\phi}\times(\del\Delta^n\hookrightarrow\Delta^n) : n\ge 0, H\supset K\xrightarrow{\phi} G\big\};
			\end{equation*}
			as $E\Inj(A,\omega)\times_H{-}$ is cocontinuous, it therefore suffices to prove the lemma for each of these generating maps. However, $E\Inj(A,\omega)^\mu\times_H (G\times H)/\Gamma_{K,\phi}\times(\del\Delta^n\hookrightarrow\Delta^n)$ simply agrees with the generating cofibration $E\Inj(A,\omega)\times_{\phi|_K} G\times(\del\Delta^n\hookrightarrow\Delta^n)$ up to isomorphism.
		\end{proof}
	\end{lem}

 	\begin{proof}[Proof of Theorem~\ref{thm:*-alg-model}]
		The symmetric monoidal category $(\cat{$\bm{E\mathcal M}$-$\bm G$-SSet}^\mu,\boxtimes)$ is closed by Corollary~\ref{cor:boxprod-closed}. Moreover, the unit axiom is immediate from Lemma~\ref{lem:boxprod-homotopical}.

		For the pushout product for cofibrations, we may restrict to generating cofibrations, where we compute using Example~\ref{ex:boxtimes-Inj} that for all $n_1,n_2\ge0$, all finite groups $H_1,H_2$, every countable faithful non-empty $H_1$-set $A_1$ and $H_2$-set $A_2$, and all homomorphisms $\phi_i\colon H_1\to G$ ($i=1,2$)
 		\begin{equation*}
 		\big(E{\Inj}(A_1,\omega)^{\mu}\times_{\phi_1}G\times(\partial\Delta^{n_1}\hookrightarrow  \Delta^{n_1})\big) \ppo\big(E{\Inj}(A_2,\omega)^{\mu}\times_\phi G\times(\partial\Delta^{n_2}\hookrightarrow \times \Delta^{n_2})\big)
 		\end{equation*}
 		agrees up to conjugation by isomorphisms with the map
 		\begin{equation}\label{eq:ppo-gen-cof-computed}
 		E{\Inj}(A_1\amalg A_2,\omega)^{\mu} \times_{H_1\times H_2}\big(G\times G\times(\partial \Delta^{n_1} \hookrightarrow\Delta^{n_1})\ppo(\partial \Delta^{n_2} \hookrightarrow\Delta^{n_2})\big)
 		\end{equation}
		where $H_1\times H_2$ acts on $A_1\amalg A_2$ in the obvious way and on $G\times G$ from the right via $\phi_1\times\phi_2$. As $A_1\amalg A_2$ is a faithful $(H_1\times H_2)$-set and $\cat{SSet}$ is a symmetric monoidal model category, Lemma~\ref{lem:more-cofibrations} shows that $(\ref{eq:ppo-gen-cof-computed})$ is a cofibration.

		Next, let $i\colon X\to X'$ be any cofibration, $j\colon Y\to Y'$ an acyclic cofibration, and consider the diagram
		\begin{equation*}
			\begin{tikzcd}
				X\boxtimes Y\arrow[dr,phantom,"\ulcorner"{very near end}]\arrow[d, "i\boxtimes Y"']\arrow[r, "X\boxtimes j"] & X\boxtimes Y'\arrow[d]\arrow[ddr, bend left=15pt, "i\boxtimes Y'"]\\
				X'\boxtimes Y\arrow[r]\arrow[drr, bend right=15pt, "X'\boxtimes j"'] & P\arrow[dr, dashed, "i\ppo j"{description}]\\
				&& X'\boxtimes Y'
			\end{tikzcd}
		\end{equation*}
		defining $i\ppo j$. Then Lemma~\ref{lem:boxprod-homotopical} shows that $X\boxtimes j$ and $X'\boxtimes j$ are $G$-global weak equivalences, while $i\boxtimes Y$ is an injective cofibration by direct inspection. Thus, the above pushout is already a homotopy pushout, and we conclude as in the proof of Proposition~\ref{prop:*-mod-quillen} that $i\ppo j$ is a weak equivalence. This finishes the proof of the pushout product axiom and hence that $\cat{$\bm{E\mathcal M}$-$\bm G$-SSet}^\mu$ is a symmetric monoidal model category.

		For the monoid axiom, we similarly observe that $X\boxtimes i$ is an injective cofibration for any cofibration in $\cat{$\bm{E\mathcal M}$-$\bm G$-SSet}^\mu$, so any relative $\mathscr C\boxtimes\text{(acyclic cofibrations)}$-cell complex is an acyclic cofibration in the injective $G$-global model structure on $\cat{$\bm{E\mathcal M}$-$\bm G$-SSet}$.

 		For the strong commutative monoid axiom for cofibrations, we may restrict to generating cofibrations by \cite[Lemma~A.1]{white}, where a similar computation as above identifies
		\begin{equation*}
			\big((E\Inj(A,\omega)^\mu\times_\phi G)\times(\del\Delta^n\hookrightarrow\Delta^n)\big)^{\ppo n}\big/\Sigma_n
		\end{equation*}
		with
		\begin{equation*}
			E\Inj(n\times A,\omega)^\mu\times_{\Sigma_n\wr H}\big( G^n \times (\del\Delta^n\hookrightarrow\Delta^n)^{\ppo n}\big)
		\end{equation*}
		where $\Sigma_n\wr H$ denotes the wreath product as usual, $\Sigma_n$ acts via permuting $n$ and the factors, and the copies of $H$ act in the evident way. As $n\times A$ is a faithful $(\Sigma_n\wr H)$-set (this uses $A\not=\emptyset$), Lemma~\ref{lem:more-cofibrations} then shows that this is a cofibration.

		For the acyclicity part, we observe that the generating cofibrations of Proposition~\ref{prop:*-mod-quillen} are maps between cofibrant objects, so \cite[Corollary~2.7]{bar} shows that there also exists a set $J$ of generating \emph{acyclic} cofibrations consisting of maps between cofibrant objects. By \cite[Corollaries~10 and~23]{gor}, it therefore suffices to show that for any acyclic cofibration $j\colon X\to Y$ of cofibrant objects the map $j^{\boxtimes n}/\Sigma_n$ is a $G$-global weak equivalence. This follows in turn immediately by combining Proposition~\ref{prop:box-powers} and Lemma~\ref{lem:cofibrant-no-fixed-points} above.

		Thus, Theorem~\ref{thm:comMon} shows that the model structur exists, and that it is combinatorial, right proper, and simplicial. Moreover, Proposition~\ref{prop:cmon-left-proper} shows that the model structure is also left proper, while filtered colimits are homotopical as they are created in $\cat{$\bm{E\mathcal M}$-$\bm G$-SSet}$.

		Finally, to see that $\incl\colon\cat{$\bm G$-ParSumSSet}\rightleftarrows\cat{$\bm G$-$\bm *$-Alg}:\!(-)^\tau$ is a Quillen equivalence, it suffices to note that fibrations and weak equivalences are created in $\cat{$\bm{E\mathcal M}$-$\bm G$-SSet}^\tau$ and $\cat{$\bm{E\mathcal M}$-$\bm G$-SSet}^\mu$, respectively, so Proposition~\ref{prop:*-mod-equiv} immediately implies that it is a Quillen adjunction, that the left adjoint preserves and reflects weak equivalences, and that for every fibrant commutative $G$-$*$-algebra $A$ the counit $A^\tau\hookrightarrow A$ is a $G$-global weak equivalence. The claim follows immediately.
	\end{proof}

	\begin{rem}\label{rk:operads-again}
		Combining the above with \cite[Theorem~2.3.1]{lenzGglobal}, the homotopy theory of commutative $G$-$*$-algebras is also equivalent to the homotopy theory of \emph{$G$-global special $\Gamma$-spaces}, while \cite[Theorem~5.27]{lenzOperads} shows that it is equivalent to the homotopy theory of algebras over any so-called \emph{twisted $G$-global $E_\infty$-operad}. However, the original proof of the equivalence between twisted $G$-global $E_\infty$-algebras and $G$-parsummable simplicial sets in \cite{lenzOperads} is somewhat subtle and indirect; in particular, it relies on an $\infty$-categorical monadicity argument to lift the unstable comparison from Theorem~\ref{thm:tame-vs-all}.

		In contrast to that, Remark~\ref{rk:*-alg-vs-operadic-alg} allows for a very easy comparison between commutative $G$-$*$-algebras and $G$-$\mathcal I$-algebras: the inclusion and its right adjoint $(-)^\mu$ define mutually inverse equivalences between the homotopy categories. Together with Theorem~\ref{thm:*-alg-model} this then gives a new, more elementary proof of the equivalence between the $G$-global homotopy theory of $G$-parsummable simplicial sets and of $\mathcal I$-algebras in $\cat{$\bm G$-SSet}$ (where $\mathcal I$ carries the trivial $G$-action). Moreover, if we identify $G$ with a universal subgroup of $\mathcal M$, then we can make $\mathcal I$ into a \emph{genuine $G$-$E_\infty$-operad} in the sense of \cite{guillou-may} via the conjugation $G$-action, and the above together with \cite[Lemma~4.23]{lenzOperads} then gives a direct equivalence between the $G$-equivariant homotopy theory of $G$-parsummable simplicial sets and of genuine $G$-$E_\infty$-algebras.
	\end{rem}

	%
	% BEGIN staralg.bbl
	%
	\providecommand{\bysame}{\leavevmode\hbox to3em{\hrulefill}\thinspace}
	\providecommand{\MR}{\relax\ifhmode\unskip\space\fi MR }
	\providecommand{\MRhref}[2]{%
	  \href{http://www.ams.org/mathscinet-getitem?mr=#1}{#2}
	}
	\providecommand{\href}[2]{#2}
	
	%
	% END staralg.bbl
	%
\end{document}